\documentclass[letterpaper, conference,10pt]{IEEEtran}  
\IEEEoverridecommandlockouts
\pdfoutput=1
\usepackage{amsmath,amssymb,bbm,graphicx}
\usepackage[utf8]{inputenc,xcolor}
\usepackage{url}
\usepackage{subfigure}
\usepackage{amsthm}
\newtheorem{definition}{\bfseries Definition}
\newtheorem{proposition}{\bfseries Proposition}
\newtheorem{example}{\bfseries Example}
\newtheorem{theorem}{\bfseries Theorem}

\newtheorem{lemma}{\bfseries Lemma}

\newtheorem{remark}{\bfseries Remark}

\allowdisplaybreaks[4]
\newcommand{\vect}[1]{\boldsymbol{\mathbf{#1}}}
\newcommand{\mat}[1]{\begin{bmatrix} #1 \end{bmatrix}}

\newcommand{\R}{\mathbb{R}}

\newif\ifdraft
\drafttrue

\title{Optimal Control for  Kinematic Bicycle Model with Continuous-time Safety Guarantees: A Sequential Second-order Cone Programming Approach}

\author{Victor Freire,  Xiangru Xu\thanks{Victor Freire and Xiangru Xu are with the Department of Mechanical Engineering, University of Wisconsin-Madison,
        Madison, WI, USA. Email: 
        {\tt\small \{victor.freiremelgizo,xiangru.xu\}@wisc.edu}.}}

\begin{document}
\maketitle
\thispagestyle{plain}
\pagestyle{plain}

\begin{abstract}
The optimal control problem for the kinematic bicycle model is considered where the trajectories are required to satisfy the safety constraints in the continuous-time sense. Based on the differential flatness property of the model, necessary and sufficient conditions in the flat space are provided to guarantee safety in the state space. The optimal control problem is relaxed to the problem of solving three second-order cone programs (SOCPs) sequentially, which find the safe path, the trajectory duration, and the speed profile, respectively.
Solutions of the three SOCPs together provide a  sub-optimal solution to the original optimal control problem. Simulation examples and comparisons with state-of-the-art optimal control solvers are presented to demonstrate the effectiveness of the proposed approach.
\end{abstract}

\section{Introduction}
The vehicle motion planning problem is well-studied. However, fast algorithms tend to lack formal safety guarantees while robust and safe planning methods are usually slow, making them unfit for real-time implementation. The search for a fast and safe planning algorithm is key to achieving provably safe driving autonomy \cite{shi_principles_2020,xu18correctness}.


The literature addressing the motion planning problem for vehicles is vast. 
A common approach is the spatio-temporal division of the problem. 
On the one hand, most pathfinding (i.e., spatial) algorithms parse the configuration space in search of minimum length or minimum curvature sequences or functions \cite{karaman2011sampling, elbanhawi2015randomized, zhang2018hybrid}. For example, in \cite{elbanhawi2015randomized}, the authors used RRT and B-spline curves to explore the space and generate kinodynamically feasible paths; 
however, their approach requires each RRT sampling step to verify kinodynamic feasibility, resulting in increased computation times. 
In \cite{zhang2018hybrid}, a multi-layer planning framework was proposed where  the pathfinding layer uses sampling techniques to modify a global path for obstacle avoidance;  however, they rely on nonconvex minimization of the path's curvature to enforce kinodynamic feasibility. On the other hand, speed profile optimization (i.e., temporal) algorithms usually focus on minimum time and maximum rider comfort while navigating a given path \cite{verscheure2009time, lipp2014minimum,zhu2015convex, liu2017speed}. For example, 
\cite{verscheure2009time} showed that a minimum-time objective function can be reformulated in terms of the path parameter, and the problem was  generalized  for certain classes of systems in \cite{lipp2014minimum}.  
In \cite{zhu2015convex}, the authors used the spatio-temporal separation to iteratively optimize the path and the speed profile; 
both optimization subproblems are convex but there is ambiguity in the stopping criteria.

\begin{figure}[!t]
\centering
\includegraphics[width=8.5cm]{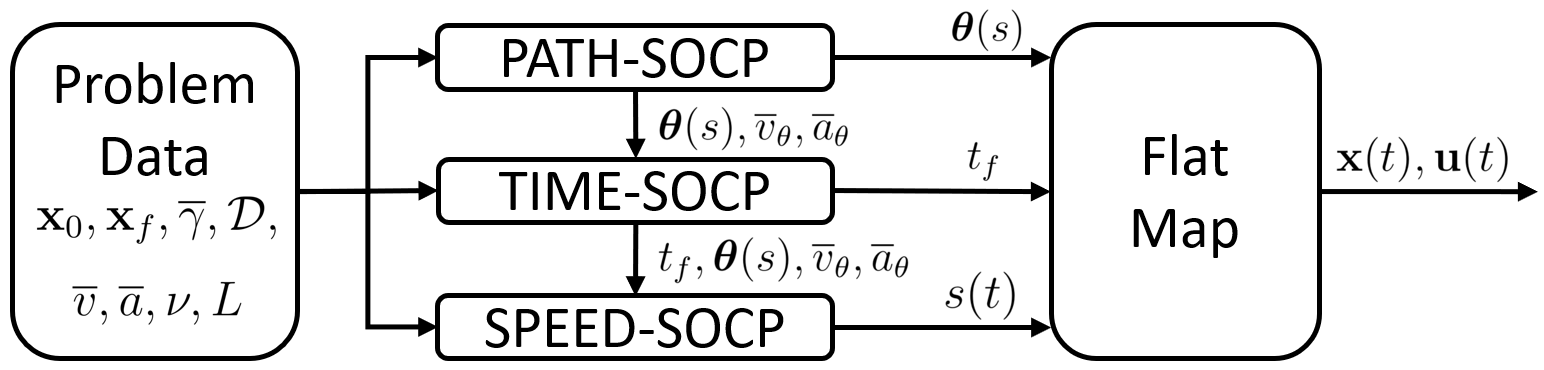}
\caption{Architecture of the sequential SOCP approach. \eqref{A-SOCP} finds a safe path $\vect{\theta}(s)$, \eqref{B-SOCP} finds an appropriate trajectory duration $t_f$, and \eqref{C-SOCP} finds a safe speed profile $s(t)$. These quantities parameterize the flat outputs which can be converted into a state-space trajectory $\vect{x}(t)$ and $\vect{u}(t)$ by the differential flatness property of the kinematic bicycle model, where  $\vect{x}(t)$ and $\vect{u}(t)$ is a guaranteed feasible point of \eqref{OPT-CTRL} by Theorem \ref{main}.}
\label{fig:framework}
\end{figure}


The kinematic bicycle model is a widely used model which captures the nonholonomic constraint present in actual vehicle dynamics and has been used in many optimal control problems for vehicle motion planning. 
For example,  \cite{polack2018guaranteeing}  studied the consistency of using the 3 DOF kinematic bicycle model for motion planning by comparing its results with a 9 DOF
model; 
\cite{kong2015kinematic} formulated an MPC problem based on the kinematic bicycle model, but the resulting problem is nonconvex and provides no safety guarantees in the continuous-time sense;  
\cite{liniger2017racing} demonstrated the high performance of stochastic MPC using the nonlinear bicycle model in a miniature racing environment, but their safety guarantees are given in the probabilistic sense.

In this work, we formulate the vehicle motion planning problem as an optimal control problem using the kinematic bicycle model. In order to ensure that the trajectories satisfy the safety constraints in the continuous-time sense, we propose necessary and sufficient conditions in flat space to guarantee safety in the state space,  based on the differential flatness property of the kinematic bicycle model.
We use spatio-temporal separation and the convexity properties of B-splines to relax the original optimal control problem into three sequential SOCPs yielding a path, a trajectory duration, and a speed profile. 
We show that the SOCP solutions constitute a suboptimal system trajectory to the original optimal control problem. Notably, SOCPs are a special type of convex optimization problem for which efficient solvers exist, which makes the proposed approach suitable for real-time and embedded applications.
Furthermore, the trajectories generated are proven to satisfy rigorously, in the continuous-time sense, the state and input safety constraints including maximum steering angle, position constraints, maximum velocity and maximum acceleration.

The remainder of the paper is organized as follows: Section \ref{sec:Preliminaries} describes the kinematic bicycle model, B-splines, second-order programming and introduces the optimal control problem considered; Section \ref{sec:Safety} provides necessary and sufficient conditions in flat space which guarantee safety in state space; Section \ref{sec:Convex} presents sufficient but convex relaxations to the previous conditions and formulates the three SOCPs; Section \ref{sec:Examples} demonstrates the proposed framework in examples and comparisons; Finally, Section \ref{sec:Conclusions} concludes the paper.

\section{Preliminaries \& Problem Statement} \label{sec:Preliminaries}
\subsection{Kinematic Bicycle  Model} \label{sec:Bike_Model}
The kinematic bicycle model is a commonly used, simple model which captures the nonholonomic constraint present in most wheeled vehicles, and it can be expressed as \cite{polack2017kinematic}: 
\begin{align}\label{dyn_model}
    \dot{\vect{x}}(t) &= f\big(\vect{x}(t)\big) + g\big(\vect{x}(t)\big)\vect{u}(t), 
\end{align}
where $f(\vect{x}) = \mat{v\cos\psi & v\sin\psi& 0&0}^T$, $g(\vect{x}) = \mat{\vect{0}_{2\times 2} & I_2}^T$,  $\vect{x} = (x, y, v, \psi)^T$ and $\vect{u} = (\dot{v}, \dot{\psi})^T$ are the state and input vectors, respectively,  $(x,y)$ is the position of the rear wheel, $v$ is the magnitude of the velocity vector and $\psi$ is the heading with respect to the inertial frame's $x$-axis (see Figure \ref{fig:coords}).
The front wheel   steering angle $\gamma$ is also a relevant quantity and is given by $\gamma = \arctan(L\dot{\psi}/ v)$, where $L>0$ is the wheelbase length.

The kinematic bicycle model \eqref{dyn_model} is a special case of the classical \emph{n-cart} system and is known to be \emph{differentially flat} \cite{rouchon1993flatness}. By choosing flat outputs as $\vect{y} = (x,y)^T$, the state and input can be expressed as functions of $\vect{y}$ and its derivatives:
\begin{equation}\label{flat_map}
    \vect{x}=\Phi(\vect{y},\dot{\vect{y}}), \quad 
    \vect{u} = \Psi(\vect{y},\dot{\vect{y}},\ddot{\vect{y}}),
\end{equation}
where the flat maps $\Phi$ and $\Psi$ are given as in \cite{rouchon1993flatness}:
\begin{equation*}
     v \!=\!\|\dot{\vect{y}}\|_2, \;
    \psi \!=\! \arctan(\dot{y}/\dot{x}), \;
    \dot{v} \!=\! \dot{\vect{y}}^T \ddot{\vect{y}}/v, \;
    \dot{\psi} \!=\!  (\ddot{y}\dot{x}\!-\!\ddot{x}\dot{y})/v^2.
\end{equation*}
Generating a trajectory for differentially flat systems reduces to finding a sufficiently smooth flat output trajectory \cite{van1998real}. In the case of system \eqref{dyn_model}, the flat trajectory $\vect{y}(t)$ needs to be at least twice-differentiable.
\begin{figure}[!hb]
    \centering
    \includegraphics[width=7cm]{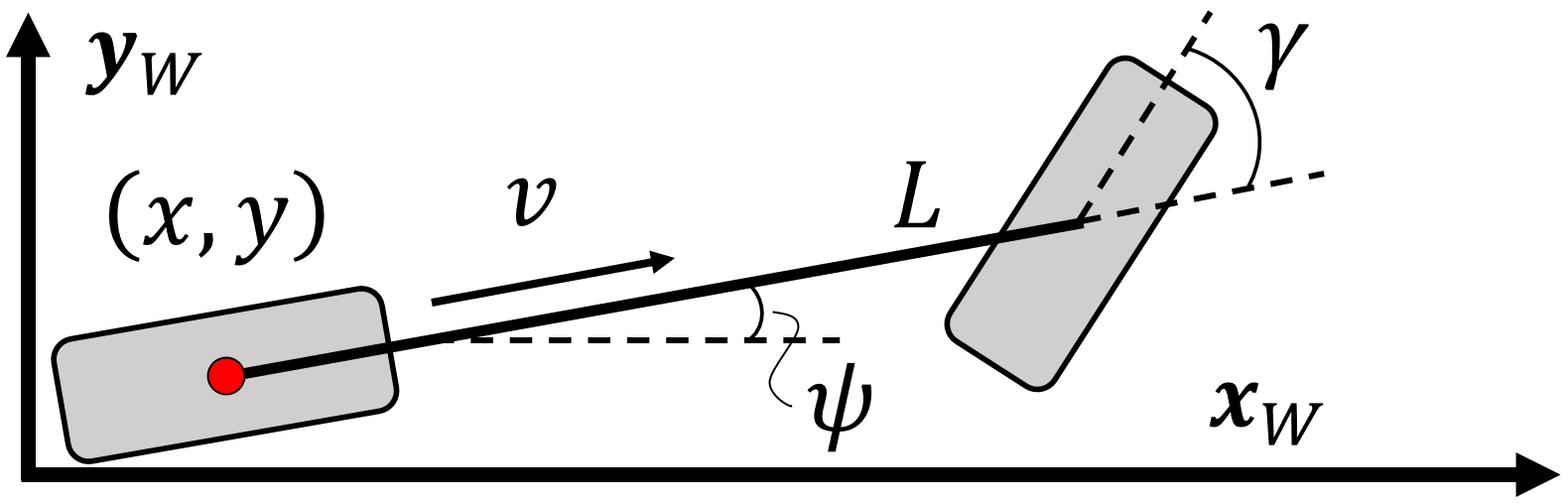}
    \caption{Kinematic bicycle model and world (inertial) coordinate frame.}
    \label{fig:coords}
\end{figure}

\subsection{B-Spline Curves} \label{sec:bsplines}



B-splines are common in trajectory generation \cite{elbanhawi2015randomized,  stoican2017constrained,freire2021flatness}. 
A $d$-th degree \emph{B-spline basis} $\lambda_{i,d}(t)$ with $d\in \mathbb{Z}_{>0}$ is defined over a given \emph{knot vector} $\vect{\tau} = (\tau_0, \ldots, \tau_{\nu})^T$ satisfying $\tau_i \leq \tau_{i+1}$ for $i = 0, \ldots, v-1$ and it can be computed recursively by the Cox-de Boor recursion formula \cite{de1978practical}.
Additionally, we consider the \emph{clamped, uniform B-spline basis}, which is defined over knot vectors satisfying:
\begin{subequations} \label{clamped_uniform}
\begin{align}
    \text{(clamped)} \quad &\tau_0=\ldots=\tau_d, \quad \tau_{\nu-d}=\ldots=\tau_{\nu},\label{clamped}\\
    \text{(uniform)}\quad &\tau_{d+1}-\tau_{d}=\ldots =\tau_{N+1}-\tau_{N},\label{uniform}
\end{align}
\end{subequations}
where $N = \nu-d-1$. A $d$-th degree \emph{B-spline curve} $\vect{s}(t)$ is a $m$-dimensional parametric curve  built by linearly combining \emph{control points} $\vect{p}_i\in\mathbb{R}^m (i=0,\ldots,N)$ and B-spline bases of the same degree. Noting $\vect{s}^{(0)}(t) = \vect{s}(t)$, we generate a B-spline curve and its $r$-th order derivative by:
\begin{equation}
    \vect{s}^{(r)}(t) = \sum_{i=0}^{N}\vect{p}_i\vect{b}_{r,i+1}^T\vect{\Lambda}_{d-r}(t) = PB_r\vect{\Lambda}_{d-r}(t),
    \label{bspline_curve}
\end{equation}
where the control points are grouped into a matrix $P = (\vect{p}_0,\dots, \vect{p}_N) \in\R^{m\times (N+1)}$,
the basis functions are grouped into a vector $\vect{\Lambda}_{d-r}(t) = \big(\lambda_{0,d-r}(t),\dots ,\lambda_{N+r,d-r}(t)\big)^T\in\R^{N+r+1}$, 
and $\vect{b}_{r,j}^T$ is the $j$-th row of a time-invariant matrix $B_{r}\in\mathbb{R}^{(N+1)\times(N+r+1)}$  constructed as $B_r = M_{d,d-r}C_r$,
where matrices $M_{d,d-r}\in\mathbb{R}^{(N+1)\times(N-r+1)}$ and $C_r\in\mathbb{R}^{(N-r+1)\times(N+r+1)}$ are defined in \cite{freire2021flatness,suryawan2012constrained}.

\begin{definition}\label{virtual_control_points} \emph{\cite{freire2021flatness}} The columns of $P^{(r)} \triangleq PB_{r}$ are called the \emph{$r$-th order virtual control points} (VCPs) of $\vect{s}^{(r)}(t)$ and denoted as $\vect{p}_i^{(r)}$ where $i=0,1,\ldots,N+r$, i.e., 
\begin{equation}\label{virtual_cp_eq}
    P^{(r)} = PB_{r} = \mat{\vect{p}_0^{(r)} & \dots & \vect{p}_{N+r}^{(r)}}.
\end{equation}
\end{definition}

B-splines have some nice properties such as continuity, convexity, and local support. The following result from \cite{freire2021flatness} ensures continuous-time set inclusion by using such properties.

\begin{proposition} \label{inclusion_prop}
\emph{\cite{freire2021flatness}} Given a convex set $\mathcal{S}$ and the $r$-th derivative of a clamped B-spline curve $\vect{s}^{(r)}(t)$ defined as in \eqref{bspline_curve}, if $\vect{p}_{j}^{(r)} \in \mathcal{S}$, with $j=r,\ldots,N$ holds, then $\vect{s}^{(r)}(t) \in \mathcal{S}, \ t\in[\tau_0,\tau_{\nu})$. Furthermore, if $\vect{p}^{(r)}_j \in S$, with $j=i-d+r,\dots,i$ and $i\in\{d,\ldots,N\}$ holds, then $\vect{s}^{(r)}(t)\in \mathcal{S}, \ t\in[\tau_i,\tau_{i+1})$.
\end{proposition}

\subsection{Second-order Cone Constraints} \label{SOC_OPT} 
\emph{Second-order cone programs} (SOCPs) are  convex optimization problems of the following form \cite{alizadeh2003second, lobo1998applications}:
\begin{align} 
    \min \quad & \vect{f}^T\vect{x}\label{SOCP_form}\\
    \text{s.t.}\quad &\|A_i\vect{x}+\vect{b}_i\|_2\leq \vect{c}_i^T\vect{x}+d_i, \quad i=1,\ldots,m,\label{SOCP_form2}
\end{align}
where $\vect{x}\in\mathbb{R}^n$ is the decision variable, $\vect{f}\in\mathbb{R}^n$, $A_i\in\mathbb{R}^{(n_i-1)\times n}$, $\vect{b}_i\in\R^{n_i-1}$, $\vect{c}_i\in\R^{n}$ and $d_i\in\R$. The constraint shown in \eqref{SOCP_form2} is the \emph{second-order cone} (SOC) constraint. 
SOCPs can be seen as a generalization of more specialized types of convex optimization problems such as linear programs (LPs), quadratic programs (QPs) and convex, quadratically constrained quadratic programs (QCQPs) \cite{alizadeh2003second,lobo1998applications}. SOCPs can be solved in polynomial time by interior-point methods, and specialized SOCP solvers also exist \cite{mosek}.

\subsection{Problem Statement}
The motion planning problem of a car is formulated as the following constrained optimal control problem with variable horizon:
\begin{subequations}
\begin{align}  \label{OPT-CTRL}\tag{OPT}
    \min\limits_{\vect{x}(\cdot),\vect{u}(\cdot),t_f} \quad & \nu t_f + \int_{0}^{t_f} L(\vect{x}(t),\vect{u}(t)) \; \mathrm{d}t \\
     \text{s. t.}\quad & \dot{\vect{x}}(t) = f(\vect{x}(t)) + g(\vect{x}(t))\vect{u}(t),\label{OPT-CTRL1}\\
     &\vect{x}(0) = \vect{x}_0, \quad \vect{x}(t_f) = \vect{x}_f, \label{OPT-CTRL2}\\
    & 0\leq v(t) \leq \overline{v}, \quad \forall t\in[0,t_f], \label{OPT-CTRL3}\\
     &|\dot{v}(t)|\leq \overline{a},\quad \forall t\in[0,t_f], \label{OPT-CTRL4}\\
     &|\gamma(t)|\leq \overline{\gamma}, \quad \forall t \in[0,t_f], \label{OPT-CTRL5}\\
        & \vect{r}(t) \in \mathcal{D},\quad \forall t\in[0,t_f], \label{OPT-CTRL6}
\end{align}
\end{subequations}
where $f$ and $g$ are defined in the kinematic bicycle model \eqref{dyn_model}, $\vect{x}_0$ and $\vect{x}_f$ are given initial and final states, respectively, $\overline{v} > 0 $ is the speed limit, $\mathcal{D} \subseteq\mathbb{R}^2$ is the obstacle-free region, $\vect{r}\triangleq(x,y)^T$ is the position vector, $\overline{a}> 0$ is the maximum acceleration and braking and $0<\overline{\gamma}< \pi/2$ is the steering angle limit. The Lagrange cost functional $L:\mathbb{R}^{4}\times\mathbb{R}^2\rightarrow \mathbb{R}$ is chosen to promote smoothness properties for the trajectory, and the parameter $\nu > 0$ encodes the tradeoff between time-optimality and smoothness. The optimal control problem \eqref{OPT-CTRL} is generally non-convex and  computationally demanding to solve in real-time. Furthermore, constraints \eqref{OPT-CTRL3}-\eqref{OPT-CTRL6} are difficult to satisfy strictly in the continuous-time sense.

In this paper, we solve \eqref{OPT-CTRL} by proposing a sequential SOCP approach, which guarantees that the constraints \eqref{OPT-CTRL3}-\eqref{OPT-CTRL6} are rigorously satisfied in the continuous-time. 
While we won't compromise on safety (feasibility), we will trade optimality for an increase in computational efficiency. Our approach leverages the differential flatness property of the bicycle model and parameterizes flat outputs using a pair of convoluted B-spline curves whose convexity properties will allow us to verify such constraints formally in continuous-time. We consider a separation between space ($\mathbb{R}^2$) and time to first find a \emph{path} with desirable properties; then, we use these properties to find a \emph{speed profile} for navigating it. Convoluting the \emph{path} with its \emph{speed profile} results in the flat output trajectory required to recover the state-space trajectory. 
\section{Safety Constraints Satisfaction in Flat Space} \label{sec:Safety}
In this section, we provide necessary and sufficient conditions  on the flat output trajectory $\vect{y}(t)$ that can guarantee continuous-time safety in the state-space. 




Consider a \emph{path} $\vect{\theta}(s)\triangleq \big(x(s),y(s)\big)^T\in C^2: [0,1]\rightarrow \R^2$ and a \emph{speed profile} $s(t)\in C^2: [0,t_f]\rightarrow[0,1]$, where $C^2$ is the set of functions whose derivatives, up to $2$nd order, exist and are continuous. The path and speed profile completely define the flat output and its derivatives \cite{pedrosa2003point}:
\begin{subequations}
\begin{align} \label{flat_parameterization}
    \vect{y}(t) &= \vect{\theta}\big(s(t)\big), \\
    \dot{\vect{y}}(t) &= \dot{s}(t)\vect{\theta}'\big(s(t)\big), \label{flat_parameterization_dot}\\
    \ddot{\vect{y}}(t) &= \ddot{s}(t)\vect{\theta}'\big(s(t)\big) + \dot{s}^2(t)\vect{\theta}''\big(s(t)\big), \label{flat_parameterization_ddot}
\end{align}
\end{subequations}
where $\vect{\theta}'\big(s(t)\big)$ denotes differentiation of $\vect{\theta}$ with respect to $s$ and taking values at $s(t)$, and similarly for $\vect{\theta}''\big(s(t)\big)$. 

\begin{remark}
The parameterization shown in \eqref{flat_parameterization} provides a number of benefits. First, the map \eqref{flat_map} has singularities if $\|\dot{\vect{y}}(t)\|_2 = 0$ for some $t$ (i.e., $\dot{x}(t)=\dot{y}(t)=0$); however, the convoluted parameterization shown in \eqref{flat_parameterization} allows one to avoid the singularity even in zero-speed situations \cite{martin2003flat}. 
Second, we can now consider the safety of the path $\vect{\theta}$, such as obstacle avoidance and steering angle constraints,  independently of the speed profile $s$ chosen later; that is, the parameterization makes spatial constraints independent from temporal constraints.
\end{remark}
In the following, we will overload the notation of the flat map \eqref{flat_map} as $\Phi(t)  \triangleq  \Phi\big(\vect{y}(t),\dot{\vect{y}}(t)\big) = \vect{x}(t)$, $\Psi(t) \triangleq \Psi\big(\vect{y}(t),\dot{\vect{y}}(t),\ddot{\vect{y}}(t)\big) = \vect{u}(t)$, 
where $\vect{y}(t)$ and its derivatives are parameterized in \eqref{flat_parameterization}-\eqref{flat_parameterization_ddot}.
\subsection{Path Safety}
Define the \emph{steering angle safety set} and \emph{drivable safety set}, respectively, as:
\begin{align}
\mathcal{S}_{\gamma} &\triangleq  \{(\vect{x},\vect{u}) \in \mathbb{R}^{4}\times \mathbb{R}^2: |\gamma|\leq\overline{\gamma}< \pi/2\}, \label{safe_set_gamma}\\
\mathcal{S}_{D} &\triangleq  \{\vect{x} \in \mathbb{R}^{4}: \vect{r}=(x,y)^T\in \mathcal{D}\}, \label{safe_set_obs}
\end{align}
with $\gamma$ the steering angle and $\mathcal{D} \subseteq \mathbb{R}^2$ the obstacle-free space.

\begin{lemma}
\label{ncvx_gamma}
The state-space trajectory $\big(\Phi(t),\Psi(t)\big) \in \mathcal{S}_{\gamma}$ for all $t\in[0,t_f]$ if and only if
\begin{equation}
\label{ncvx_gamma_cst}
    \|\vect{\theta}'(s)\times\vect{\theta}''(s)\|_2/\|\vect{\theta}'(s)\|_2^3 \leq \tan\overline{\gamma}/L, \quad \forall s\in[0,1].
\end{equation}
\end{lemma}

\begin{proof}
It can be shown from the expression of $\gamma$ given in Section \ref{sec:Bike_Model} and the considered parameterization \eqref{flat_parameterization} that $\gamma(t) = \arctan\big({L(y''\small(s\small)x'\small(s\small) - x''\small(s\small)y'\small(s\small))}/\|\vect{\theta}'\small(s\small)\|_2^3\big)$. 
Furthermore, over the range $\gamma\in(-\pi/2,\pi/2)$, $\tan|\gamma| = |\tan\gamma|$. Therefore, $\tan|\gamma|/L = {\|\vect{\theta}'(s)\times\vect{\theta}''(s)\|_2}/{\|\vect{\theta}'(s)\|_2^3}$.
Observing that $
\tan(|\gamma|) \leq\tan(\overline{\gamma}), \gamma \in (-\pi/2,\pi/2)\iff|\gamma|\leq \overline{\gamma}<\pi/2$,
the conclusion follows immediately.
\end{proof}

\begin{lemma}\label{ncvx_obs}
The state-space trajectory $\Phi(t)\in\mathcal{S}_D$ for all $t\in[0,t_f]$ if and only if
\begin{equation}\label{ncvx_obs_cst}
    \vect{\theta}(s) \in \mathcal{D}, \quad \forall s\in[0,1]. 
\end{equation}
\end{lemma}

\begin{proof}
Recall that the speed profile $s\in C^2:[0,t_f]\rightarrow[0,1]$. Thus, the condition $\vect{y}(t) = \vect{\theta}(s(t))\in\mathcal{D}$ must hold for all $t\in[0,t_f]$. The conclusion follows by the definition of $\mathcal{S}_D$.
\end{proof}

\subsection{Speed Profile Safety}
Define the \emph{forward speed safety set} and \emph{linear acceleration safety set}, respectively, as:
\begin{align}
     \mathcal{S}_{v} &\triangleq \{\vect{x}\in\mathbb{R}^4: 0 \leq v \leq \overline{v}\}, \label{safe_set_v}\\
     \mathcal{S}_{\dot{v}} &\triangleq \{\vect{u}\in\mathbb{R}^2: |\dot{v}| \leq \overline{a}\}, \label{safe_set_v_dot}
\end{align}
with $\overline{v}$ and $\overline{a}$ the speed and acceleration bounds, respectively.

\begin{lemma} \label{ncvx_v}
Let $\vect{\theta}$ be a path. The state-space trajectory $\Phi(t) \in \mathcal{S}_v$ for all $t\in[0,t_f]$ if and only if 
    \begin{equation}\label{ncvx_v_cst}
        \dot{s}(t) \geq 0 , \quad \dot{s}(t)\|\vect{\theta}(s(t))\|_2 \leq \overline{v}, \quad \forall t \in [0,t_f].
    \end{equation}
\end{lemma}
\begin{proof}
The conclusion follows from the flat map $\eqref{flat_map}$ describing the state $v$ and the parameterization of $\dot{\vect{y}}(t) $ given in \eqref{flat_parameterization_dot}.
\end{proof}

\begin{lemma} \label{ncvx_v_dot}
Let $\vect{\theta}$ be a path. The input trajectory $\Psi(t)\in \mathcal{S}_{\dot{v}}$ for all $t\in [0,t_f]$ if and only if
\begin{equation} \label{ncvx_v_dot_cst}
|a_t(t) + a_n(t)|\leq \overline{a}, \quad \forall t\in[0,t_f],
\end{equation}
where $a_t\triangleq \ddot{s}\|\vect{\theta}'(s)\|_2$ and $a_n \triangleq \dot{s}^2\big(\vect{\theta}'(s)\cdot \vect{\theta}'' (s)\big)/\|\vect{\theta}'(s)\|_2$.
\end{lemma}  

\begin{proof}
Differentiating $v(t) = \dot{s}(t)\|\vect{\theta}'\big(s(t)\big)\|_2$ with respect to time we have $\dot{v}(t) = a_t(t) + a_n(t)$. The conclusion follows immediately by the definition of $\mathcal{S}_{\dot{v}}$.
\end{proof}

\subsection{Flattened Optimal Control Problem}
Consider now the following functional optimization problem:
\begin{align} \label{FLAT-OPT}
    \min\limits_{\vect{\theta}(\cdot), s(\cdot), t_f} \quad &  \nu t_f + \int_{0}^{t_f} L\big(\Phi(t),\Psi(t)\big) \; \mathrm{d}t, \tag{FLAT-OPT} \\
    \text{s. t.} \quad & \Phi(0) = \vect{x}_0, \quad \Phi(t_f) = \vect{x}_{f}, \nonumber\\ \quad & \text{\eqref{ncvx_gamma_cst}, \eqref{ncvx_obs_cst}, \eqref{ncvx_v_cst} and \eqref{ncvx_v_dot_cst} hold}, \nonumber
\end{align}
where $\vect{\theta}:[0,1]\rightarrow\mathbb{R}^2$ and $s:[0,t_f]\rightarrow [0,1]$.
\begin{proposition} \label{equiv_cor}
If the duration $t_f$, \emph{path} $\vect{\theta}$ and \emph{speed profile} $s$ is a solution of \eqref{FLAT-OPT}, then the corresponding trajectory given by $\vect{x}(t) = \Phi(t)$ and $\vect{u}(t) = \Psi(t)$ is a solution of \eqref{OPT-CTRL}.
\end{proposition}
\begin{proof}
The constraints on the initial and final states hold: $\Phi(0) = \vect{x}(0) = \vect{x}_0$ and $\Phi(t_f) = \vect{x}(t_f) = \vect{x}_f$.
By the definition of each safety set $\mathcal{S}_{\gamma}, \mathcal{S}_D,  \mathcal{S}_v$ and $\mathcal{S}_{\dot{v}}$ and by Lemmas \ref{ncvx_gamma}, \ref{ncvx_obs}, \ref{ncvx_v} and \ref{ncvx_v_dot}, it follows that the state-space trajectory $\vect{x}(t)$ and $\vect{u}(t)$ satisfies all safety constraints in \eqref{OPT-CTRL}. The differential constraint in \eqref{OPT-CTRL} is automatically satisfied by virtue of the differential flatness property \cite{van1998real} and $C^2$ smoothness of the \emph{path} $\vect{\theta}$ and \emph{speed profile} $s$. Finally, notice that the objective functionals are identical in both problems.
\end{proof}

Proposition \ref{equiv_cor} is a particular case of the observation in \cite{ross2002pseudospectral} that for differentially flat systems, optimal control problems can be cast as functional optimization problems without differential constraints by virtue of the flatness property. Intuitively, the state differential constraint is translated into a smoothness constraint in the flat output. 
However, \eqref{FLAT-OPT} is still intractable because the problem is nonconvex and we are minimizing over functions instead of vectors. In the next section, we will let the path $\vect{\theta}$ and the speed profile $s$ be B-spline curves and optimize over their control points. We also use their convexity properties to reformulate the constraints into convex conditions with respect to their control points.

\section{Convexification of Problem \eqref{FLAT-OPT}} \label{sec:Convex}
In this section, we describe a sequential convexification approach for  \eqref{FLAT-OPT} by  splitting the problem into three sequential SOCP: the first SOCP finds a safe path $\vect{\theta}(s)$, the second SOCP finds a duration $t_f$ for the trajectory, and the third SOCP computes a safe velocity profile $s(t)$ (see Figure \ref{fig:framework}). The solution of these three convex programs together provides a feasible and possibly sub-optimal solution to  \eqref{FLAT-OPT} with rigorous continuous-time constraint satisfaction guarantees.

\begin{definition}[B-spline path]\label{B-spline Path}
A \emph{B-spline path} is a $C^2$, $2$-dimensional, $d_{\theta}$-degree B-spline curve defined as in \eqref{bspline_curve} over a clamped, uniform knot vector $\vect{\zeta}$ segmenting the interval $[0,1]$ and control points $\vect{\Theta}_j \in\mathbb{R}^2, \; j=0,\ldots,N_{\theta}$.
\end{definition}

\begin{definition}[B-spline speed profile] \label{B-spline Speed Profile}
A \emph{B-spline speed profile} is a $C^2$, $1$-dimensional, $d_s$-degree B-spline curve defined as in \eqref{bspline_curve} over a clamped, uniform knot vector $\vect{\tau}$ segmenting the interval $[0,t_f]$ and control points $p_j\in\mathbb{R},\; j = 0,\ldots,N_{s}$.
\end{definition}


\subsection{B-spline Path Optimization}
The following propositions provide convex relaxations to the conditions of Lemma \ref{ncvx_gamma} and Lemma \ref{ncvx_obs}. 
\begin{proposition} \label{cvx_gamma}
Let $\vect{\theta}(s)$ be a B-spline path and $b_{r,i,j}$ be the $(i,j)$-th entry of matrix $B_r$ defined in Section \ref{sec:bsplines}.
If there exists positive constant $\alpha >0$, column unit vector $\hat{\vect{r}}\in\mathbb{R}^2$, and variables $\underline{v}_{\theta},\overline{a}_{\theta},\beta\in\mathbb{R}$ such that the B-spline path $\vect{\theta}(s)$ satisfies the following conditions:
\begin{subequations}\label{cvx_gamma_cst}
\begin{align}
    &\hat{\vect{r}}^T\vect{\Theta}^{(1)}_{i} \geq \underline{v}_{\theta}, \quad j = 1,\ldots, N_{\theta},\label{cvx_gamma_cst_negdisc1}\\ 
    &\big\|\vect{\Theta}^{(2)}_{i}\big\|_2 \leq \overline{a}_{\theta}, \quad j = 2,\ldots,N_{\theta},\label{cvx_gamma_cst_negdisc2}\\
    &\big\|\mat{2\alpha & \frac{4\tan\overline{\gamma}}{L}\beta-1}^T\big\|_2 \leq 4\beta\tan\overline{\gamma}/L + 1,\label{cvx_gamma_cst_negdisc}\\
    &\overline{a}_{\theta} \leq \alpha\underline{v}_{\theta} - \beta, \quad 
    \beta, \underline{v}_{\theta} \geq 0,
    \end{align}
\end{subequations}
then the state-space trajectory $\big(\vect{x}(t),\vect{u}(t)\big)\in\mathcal{S}_{\gamma},$  $\forall t\in[0,t_f]$ where $\mathcal{S}_{\gamma}$ is the \emph{steering angle safety set} defined in \eqref{safe_set_gamma}.
\end{proposition}
\begin{proof}
By Proposition \ref{inclusion_prop} conditions \eqref{cvx_gamma_cst_negdisc1}-\eqref{cvx_gamma_cst_negdisc2} imply that $\underline{v}_{\theta} \leq \|\vect{\theta}'(s)\|_2$ and $\|\vect{\theta}''(s)\|_2\leq \overline{a}_{\theta}$ for all $s\in[0,1)$. 
Continue by expanding  \eqref{cvx_gamma_cst_negdisc} to observe that $0 \geq \alpha^2 -4\beta\tan\overline{\gamma}/L\beta \triangleq \Delta_{v}$.
Notice that $\Delta_{v}$ is the discriminant of the
quadratic polynomial in $\underline{v}_{\theta}$ given by $p(\underline{v}_{\theta}) \triangleq
\frac{\tan\overline{\gamma}}{L}\underline{v}_{\theta}^2 -
\alpha \underline{v}_{\theta} + \beta$. 
  It follows from $\Delta_{v}\leq 0$ that the roots of
  $p(\underline{v}_{\theta})$ are either repeated and real, or complex
  conjugates. Therefore, the polynomial $p(\underline{v}_{\theta})$ does
  not change sign. Since  $p(0)\geq 0$ (because $\beta \geq 0 $), we must have that $p(\underline{v}_{\theta})\geq0$. In particular, we can
  now observe that
  \begin{equation*}
\overline{a}_{\theta} \!\leq\! \alpha \underline{v}_{\theta} - \beta
    \!\leq \underline{v}_{\theta}^2\tan\overline{\gamma}/L \!\implies\! \|\vect{\theta}''(s)\|_2 \!\leq\! \|\vect{\theta}'(s)\|_2^2\tan\overline{\gamma}/L
  \end{equation*}
  holds. Now multiply by $\|\vect{\theta}'(s)\|_2 \geq 0$ in both sides of the implied inequality to establish:
  \begin{equation*}
     \|\vect{\theta}'(s)\|_2^3 \tan\overline{\gamma}/L \geq \|\vect{\theta}''(s)\|_2 \|\vect{\theta}'(s)\|_2 \geq \|\vect{\theta}'(s)\times \vect{\theta}''(s)\|_2. 
  \end{equation*}
  The conclusion follows directly from Lemma \ref{ncvx_gamma}.
\end{proof}

\begin{proposition}\label{cvx_obs}
Let $\mathcal{D}\subseteq\mathbb{R}^2$ be a given SOC. If the control points $\vect{\Theta}_j$ of the B-spline path $\vect{\theta}(s)$ satisfy:
\begin{align}\label{cvx_obs_cst}
    \vect{\Theta}_j \in \mathcal{D}, \quad j = 0,\ldots, N_\theta,
\end{align}
then the state-space trajectory $\vect{x}(t)\in\mathcal{D}$ for all $t\in[0,t_f]$.
\end{proposition}
\begin{proof}
Condition \eqref{cvx_obs_cst} implies, by Proposition \ref{inclusion_prop}, that $\vect{\theta}(s) \in \mathcal{D}$ for all $s\in[0,1)$. The conclusion follows from Lemma \ref{ncvx_obs}.
\end{proof}
\begin{remark}
While the obstacle-free space $\mathcal{D}$ is generally nonconvex, the convexity assumption is easily relaxed by considering the concept of ``safe corridor'' (union of convex sets) and enforcing the conditions of Proposition \ref{cvx_obs} segment-wise instead of globally. The reader is referred to our previous work \cite{freire2021flatness} and to \cite{gao2018online,sun2020fast} for more information.
\end{remark}

For fixed values of $\alpha$ and $\hat{\vect{r}}$, the conditions of Proposition \ref{cvx_gamma} are convex and we can formulate the following SOCP to find a safe path $\vect{\theta}(s)$:
\begin{align} \label{A-SOCP}\tag{PATH-SOCP}
    \min
    \; &\int_{0}^{1}\|\vect{\theta}'''(s)\|_2^2 \; \mathrm{d}s + \overline{v}_{\theta} - \underline{v}_{\theta} + \overline{a}_{\theta}\\
    \text{s. t.} \;& \vect{\Theta}_0  = \vect{r}_0,\quad \vect{\Theta}_{N_\theta} = \vect{r}_f,\nonumber\\
    &\vect{\Theta}^{(1)}_{1} \!=\! \overline{v}_{\theta}(\cos\psi_0, \sin\psi_0)^T, \nonumber\\
    &\vect{\Theta}^{(1)}_{N_{\theta}} \!=\! \overline{v}_{\theta}(\cos\psi_f, \sin\psi_f)^T, \nonumber\\
    &\big\|\vect{\Theta}^{(1)}_{i}\big\|_2 \leq \overline{v}_{\theta}, \quad j = 1,\ldots,N_{\theta},\nonumber\\
    & \mbox{\text{\eqref{cvx_gamma_cst} and \eqref{cvx_obs_cst} hold}}, \nonumber
\end{align}
with decision variables $\beta,\overline{v}_{\theta},\underline{v}_{\theta},\overline{a}_{\theta}$ and $\vect{\Theta}_0,\ldots,\vect{\Theta}_{N_{\theta}}$, 
where $\overline{\gamma}$ is the maximum steering angle, $\vect{r}_m \triangleq (x_m, y_m)^T, \; m\in\{0,f\}$ is the initial/final position vector and $\psi_m$ is the initial/final heading angle. Note that the resulting B-spline path $\vect{\theta}$ has bounded derivatives given by $\overline{v}_{\theta}$ and $\overline{a}_{\theta}$ which will be used to guarantee the safety of speed profiles in Section \ref{sec:Speed Profile Optimization}.


\begin{remark}
Proposition \ref{cvx_gamma} contains two convex relaxations by lower-bounding a norm and a convex, quadratic function. Because of the relaxations, the size of the feasible region of \eqref{A-SOCP} depends on the choice of parameters $\alpha$ and $\hat{\vect{r}}$. We found the following heuristics worked well:
\begin{equation*}
    \alpha = \frac{2\tan\overline{\gamma}}{L}\|\vect{r}_f-\vect{r}_0\|_2, \quad
    \hat{\vect{r}} = \frac{\vect{r}_f - \vect{r}_0}{\|\vect{r}_f-\vect{r}_0\|_2}.
\end{equation*}
\end{remark}



\subsection{Temporal Optimization}
In this subsection, we find an appropriate trajectory duration
$t_f$ by considering a minimization of both $t_f$ and the magnitude of the acceleration vector.
Consider a B-spline path $\vect{\theta}(s)$ feasible in \eqref{A-SOCP}, and the  functions $b(s) \triangleq \dot{s}^2$ and $a(s) \triangleq\ddot{s}$, which must satisfy the differential condition $b'(s) = 2a(s)$.
Following \cite{verscheure2009time}, we have that
\begin{equation} \label{time_int}
    t_f = \int_{0}^{t_f}1\;\mathrm{d}t = \int_{0}^1\frac{1}{\dot{s}}\;\mathrm{d}s = \int_{0}^1b(s)^{-1/2} \; \mathrm{d}s.
\end{equation}
Purely minimizing the trajectory duration given by $t_f$ results in maximum-speed velocity profiles. We additionally minimize the acceleration to encourage trajectories with mild friction circle profiles  by choosing a Lagrange cost functional \cite{rajamani2011vehicle}: 
\begin{equation}
    L\big(\vect{x}(t),\vect{u}(t)\big) \triangleq \dot{v}^2(t) + v^2(t)\dot{\psi}^2(t)= \|\ddot{\vect{y}}(t)\|_2^2 .
\end{equation}
We can now write the objective functional entirely in terms of (and convex with respect to) the new functions $a(s)$ and $b(s)$ as follows:
\begin{equation*}
    J\big(a(s),b(s)\big) = \int_{0}^1\frac{\nu}{\sqrt{b(s)}} + \| a(s)\vect{\theta}'(s) + b(s)\vect{\theta}''(s)\|^2_2 \; \mathrm{d}s.
\end{equation*}
We follow a similar procedure as in \cite{verscheure2009time} and consider $N_t+1$ points partitioning the interval $[0,1]$ into $N_t$ uniform segments with width $\Delta s \triangleq 1/N_t$.
The discretized Lagrange cost functional becomes 
$$
L(s_i,a_i,b_i) = \| a_i\vect{\theta}'(s_i) + b_i\vect{\theta}''(s_i)\|^2_2,
$$ 
where $a_i$ and $b_i$ are the decision variables representing $a(s_i)$ and $b(s_i)$, respectively. Assuming that $a(s)$ is piecewise constant over each segment $[s_i, s_{i+1}), i = 0,\ldots,N_t-1$ where $s_i \triangleq i\Delta s$, we can exactly evaluate the integral \eqref{time_int} to avoid the case when $\dot{s} = 0$ as shown in \cite{verscheure2009time}. 

We formulate the following SOCP to obtain the trajectory duration $t_f$:
\begin{align} \label{B-SOCP}\tag{TIME-SOCP}
    \min \quad& \sum_{i=0}^{N_t-1} 2\nu\Delta sd_i + \sum_{i=0}^{N_t} L(s_i,a_i,b_i)\\
    \text{s. t.} \quad &\big\|(2c_i , b_i-1)^T\big\|_2 \leq b_i+1,\quad i = 0,\ldots,N_t,\nonumber\\
    &\big\|(2, \overline{c}_i - d_i)^T\big\|_2 \leq \overline{c}_i + d_i,\; i=0,\ldots,N_t-1,\nonumber\\
    &2\Delta sa_i = b_i - b_{i-1},\quad i=1,\ldots,N_t,\nonumber\\
    &b_0 \|\vect{\theta}'(0)\|_2^2 = v_0^2, \quad b_{N_t}\|\vect{\theta}'(1)\|_2^2 = v_f^2,\nonumber\\
    &b_i\|\vect{\theta}'(s_i)\|_2^2 \leq \overline{v}^2, \quad i = 0,\ldots,N_t,\nonumber\\
    &\Big| a_i \|\vect{\theta}'(s_i)\|_2 + b_if_i \Big|\leq \overline{a}, \quad i=0,\ldots,N_t, \nonumber\\
    &\overline{c}_i \triangleq  c_{i+1} + c_i, \quad f_i \triangleq \big(\vect{\theta}'(s_i) \cdot \vect{\theta}''(s_i)\big)/\|\vect{\theta}(s_i)\|_2, \nonumber
\end{align}
with decision variables $a_i,b_i,c_i,d_i$, $i= 0,\ldots,N_t$. The first two constraints are the SOCP embedding of \eqref{time_int} given in \cite{verscheure2009time}, the third constraint is from the differential constraint relating $a(s)$ and $b(s)$, the fourth sets initial and final speeds, the fifth ensures the speed bound is respected and the sixth ensures the acceleration bound is respected. From our assumption that the function $a(s)$ is constant over each $[s_i,s_{i+1})$ segment, we can recover the duration of each segment from the constant acceleration equation $\Delta t_i = (\sqrt{b_i} + \sqrt{b_{i-1}})/a_i,\; i = 1,\ldots, N_t.$ 
The overall duration of the trajectory is then $ t_f = \sum_{i=1}^{N_t} \Delta t_i. $

\begin{remark}
The solution of \eqref{B-SOCP} provides a safe speed profile at discrete time instances. If continuous-time safety is not critical, it suffices to stop here and retrieve the discretized state-space solution. In addition, if the desired trajectory duration $t_f$ is known, one can skip \eqref{B-SOCP} and proceed to the next SOCP after solving \eqref{A-SOCP}.
\end{remark}

\subsection{Speed Profile Optimization} \label{sec:Speed Profile Optimization}
Assume that $t_f$ is given by the solution of \eqref{B-SOCP} or specified a priori. Let $\vect{\theta}(s)$ be a B-spline path feasible in \eqref{A-SOCP} and $s(t)$ be a B-spline speed profile. The following propositions provide convex relaxations to the conditions of Lemma \ref{ncvx_v} and Lemma \ref{ncvx_v_dot}. 
\begin{proposition} \label{cvx_v}
If the condition
\begin{equation}\label{cvx_v_cst}
    0\leq \overline{v}_{\theta} p_j^{(1)} \leq \overline{v}, \quad j = 1,\ldots, N_s,
\end{equation}
holds, then the state-space trajectory $\vect{x}(t)\in\mathcal{S}_v,$ $\forall t\in[0,t_f)$, where $\mathcal{S}_v$ is the \emph{forward speed safety set} defined in \eqref{safe_set_v}.
\end{proposition}
\begin{proof}
By Proposition \ref{inclusion_prop}, \eqref{cvx_v_cst} implies that $0\leq \overline{v}_{\theta}\dot{s}(t)\leq \overline{v}$ for all $t\in[0,t_f)$. Because $   \overline{v} \geq \overline{v}_{\theta}\dot{s}(t) \geq \dot{s}(t)\|\vect{\theta}\big(s(t)\big)\|_2 = v(t) \geq 0$, 
the conclusion follows.
\end{proof}

\begin{proposition}
\label{cvx_v_dot} For any given nonnegative vectors $\overline{\vect{m}} = (\overline{m}_0,\ldots,\overline{m}_{N_s-d_s})^T, \; m\in\{\kappa,\epsilon\},$ if the following conditions
\begin{subequations}\label{cvx_v_dot_cst}
\begin{align}
    &0\leq p^{(1)}_j \leq \overline{\kappa}_{k} ,\quad j= k+1,\ldots,k+d,\\
    &-\overline{\epsilon}_{k} \leq p_j^{(2)}\leq \overline{\epsilon}_k,\quad j = k+2,\ldots,k+d,\\
    &\big\|A_{\dot{v}}(\overline{\kappa}_k, \overline{\epsilon}_k)^T  + \vect{b}_{\dot{v}}\big\|_2 \leq \mat{\overline{\kappa}_k & \overline{\epsilon}_k}\vect{c}_{\dot{v}} + d_{\dot{v}},
\end{align}
\end{subequations}
hold for all $k\in\{0,\ldots,N_s-d_s\}$, where $A_{\dot{v}} = \text{\emph{diag}}$($\sqrt{2\overline{a}_{\theta}}$, $-\overline{v}_{\theta}/\sqrt{2})$, $\vect{b}_{\dot{v}} = (0,  (\overline{a} - 1)/\sqrt{2})^T$, $\vect{c}_{\dot{v}} = (0,  -\overline{v}_{\theta}/\sqrt{2})^T$ and $d_{\dot{v}} =  (\overline{a}+1)/\sqrt{2}$,
then the trajectory $\vect{u}(t)\in\mathcal{S}_{\dot{v}}$,  $\forall t\in[0,t_f)$ where $\mathcal{S}_{\dot{v}}$ is the \emph{linear acceleration safety set} defined in \eqref{safe_set_v_dot}.
\end{proposition}
\begin{proof}
The first two conditions imply, by Proposition \ref{inclusion_prop}, that for any $k\in\{0,\ldots,N_s-d_s\}$, $0\leq\dot{s}(t)\leq \overline{\kappa}_k$ and $|\ddot{s}(t)| \leq \overline{\epsilon}_k$ hold  $\forall t \in [\tau_{k+d_s},\tau_{k+d_s+1})$.
Expanding the last condition, we determine that for any $k\in\{0,\ldots,N_s-d_s\}$, $0 \leq \overline{\kappa}^2_k\overline{a}_{\theta} + \overline{\epsilon}_{k}\overline{v}_{\theta} = |\overline{\kappa}^2_k\overline{a}_{\theta}|+|\overline{\epsilon}_{k}\overline{v}_{\theta}| \leq \overline{a}$. 
Notice that for any $k\in\{0,\ldots,N_s-d_s\}$, the following conditions hold for all $t\in[\tau_{k+d_s},\tau_{k+d_s+1})$:
\begin{align*}
    |\overline{\epsilon}_k\overline{v}_{\theta}| &\geq  |\ddot{s}\small(t\small)|\|\vect{\theta}'(s\small(t\small))\|_2 = |a_t\small(t\small)|, \\
    |\overline{\kappa}_k^2\overline{a}_{\theta}| &\geq \big|{\dot{s}^2\small(t\small) \big(\vect{\theta}'(s\small(t\small))\cdot\vect{\theta}''(s\small(t\small))\big)}/{\|\vect{\theta}'(s\small(t\small))\|}\big| = |a_n\small(t\small)|,
\end{align*}
where $a_t$ and $a_n$ are defined in \eqref{ncvx_v_dot_cst}. Therefore, by the triangle inequality $\overline{a} \geq |a_t(t)| + |a_n(t)| \geq |a_t(t) + a_n(t)|$ holds for all $t\in[\tau_{k+d_s},\tau_{k+d_s+1})$. Recalling that the knot vector $\vect{\tau}$ is clamped and uniform \eqref{clamped_uniform}, segmenting the interval $[0,t_f]$, and that the above inequality holds for all $k\in\{0,\ldots,N_s-d_s\}$, we can establish that the inequality holds, in fact, for all $t\in[0,t_f)$. The conclusion now follows by Lemma \ref{ncvx_v_dot}.
\end{proof}

We formulate the following SOCP to obtain the speed profile.
\begin{align}\tag{SPEED-SOCP}\label{C-SOCP}
    \min\limits_{\overline{\vect{\kappa}}, \overline{\vect{\epsilon}}, p_0,\ldots,p_{N_s}} \quad & \int_{0}^{t_f}\dddot{s}^2(t)\;\mathrm{d}t 
    \\
    \text{s. t.} \quad & p_0 = 0, \quad p_{N_s} = 1,\nonumber\\
    &\overline{v}_{\theta}p^{(1)}_1 = v_0, \quad \overline{v}_{\theta}p^{(1)}_{N_s} = v_f \nonumber\\
    & \mbox{\text{\eqref{cvx_v_cst} and \eqref{cvx_v_dot_cst} hold}}, \nonumber
\end{align}
where $v_0$ and $v_f$ are given initial and final speeds, respectively. 


\subsection{Safety Analysis}
The following theorem summarizes the main theoretical contributions of this paper.
\begin{theorem} \label{main}
Let $\vect{\theta}(s)$ and $s(t)$ be, respectively, a B-spline path feasible in \eqref{A-SOCP} and a B-spline speed profile feasible in \eqref{C-SOCP} with duration $t_f$.
The corresponding state-space trajectory of \eqref{dyn_model} obtained by passing the parameterized flat outputs \eqref{flat_parameterization} in terms of $\vect{\theta}$ and $s$ through the flat map \eqref{flat_map} satisfies the initial $\vect{x}(0) = \vect{x}_0$ and final $\vect{x}(t_f) = \vect{x}_f$ conditions as well as safety specifications:
\begin{equation}\label{thrm1_safety}
    \big(\vect{x}(t),\vect{u}(t)\big)\in\mathcal{S}_{\gamma}, \quad \vect{x}(t) \in \mathcal{S}_{\mathcal{D}} \cap \mathcal{S}_v, \quad \vect{u}(t) \in \mathcal{S}_{\dot{v}},
\end{equation}
$\forall t \in [0,t_f)$. Thus, $t_f$, $\vect{x}(t)$ and $\vect{u}(t)$ are feasible in \eqref{OPT-CTRL}.
\end{theorem}
\begin{proof}
The initial and final positions $(x,y)$ are satisfied because  $\vect{\theta}(0) = \vect{r}_0 = (x_0,y_0)^T$ and $\vect{\theta}(1) = \vect{r}_f = (x_f,y_f)^T$. Furthermore, $\|\vect{\theta}(s)\|_2 = \overline{v}_{\theta}$ for $s\in\{0,1\}$. Since $s(0) = 0$, $s(t_f) = 1$, $\overline{v}_{\theta}\dot{s}(0) = v_0$ and $\overline{v}_{\theta} \dot{s}(t_f) = v_f$, the flat output parameterization \eqref{flat_parameterization_dot} and the flat map \eqref{flat_map} imply that the obtained state-space trajectory $\vect{x}(t)$ satisfies the specified initial and final velocities $v_0$ and $v_f$. For the initial and final heading angles $\psi_0$ and $\psi_f$, notice that $\vect{\theta}'(m) = \big(x'(m), y'(m)\big)^T = \overline{v}_{\theta} \big(\cos\psi_{q}, \sin\psi_q\big)^T$, 
with $(m,q)\in\{(0,0),(1,f)\}$.
Since the trajectory $\psi(t)$  is given only by the path $\vect{\theta}(s)$ under parameterization \eqref{flat_parameterization}, it follows that $\psi(p) = \arctan\big({\overline{v}_{\theta}\sin\psi_q}/({\overline{v}_{\theta}\cos{\psi_q})}\big) = \psi_q$,  $(p,q) \in\{(0,0),(1,f)\}$.
Finally, \eqref{thrm1_safety} follows by Propositions \ref{cvx_gamma}, \ref{cvx_obs}, \ref{cvx_v} and \ref{cvx_v_dot}. 
\end{proof}


\section{Simulation Examples } \label{sec:Examples}
In this section, we evaluate the performance and efficiency of the proposed approach by three simulation examples. 
In all examples, we use  B-spline path $\vect{\theta}$ parameters $d_{\theta} = 4$ and $N_{\theta} = 20$, and B-spline speed profile $s$ parameters $d_{s} = 4$ and $N_{s} = 20$. For \eqref{B-SOCP}, we use $N_t = 40$.

\begin{example}\label{example1}
Consider \eqref{OPT-CTRL} with the following problem data: initial state $\vect{x}_0 = \vect{0}_{4\times 1}$, final state $\vect{x}_f = (100, 4, 0, 0)^T$, obstacle-free space $\mathcal{D} = \mathbb{R}^2$, maximum steering angle $\overline{\gamma} = 0.0044$ rad ($0.25$ degrees), maximum speed $\overline{v} = 4.2$ m/s, maximum acceleration and braking $\overline{a} = 0.6$ m/s$^2$, duration penalty factor $\nu = 1$ and wheelbase length $L = 2.601$ meters. Note that this problem requires rest-to-rest motion to be solved and, as described in previous sections, the flat map has singularities when the speed is zero. Nonetheless, our approach is able to handle this gracefully. We continue by solving the three proposed SOCP problems sequentially using YALMIP \cite{lofberg2004yalmip} and MOSEK \cite{mosek}. We then pass the resulting path and speed profile through the flat map \eqref{flat_map} to obtain corresponding state $\vect{x}(t)$ and input $\vect{u}(t)$ trajectories. The resulting state and input trajectories are shown in Figure \ref{fig:safe} along with the steering angle $\gamma(t)$. It can be seen that the speed $v$ the acceleration $\dot{v}$ and the steering angle $\gamma$ respect their bounds for all time $t \in [0,t_f]$.

\begin{figure}[htp]
    \centering
    \includegraphics[width=4.3cm]{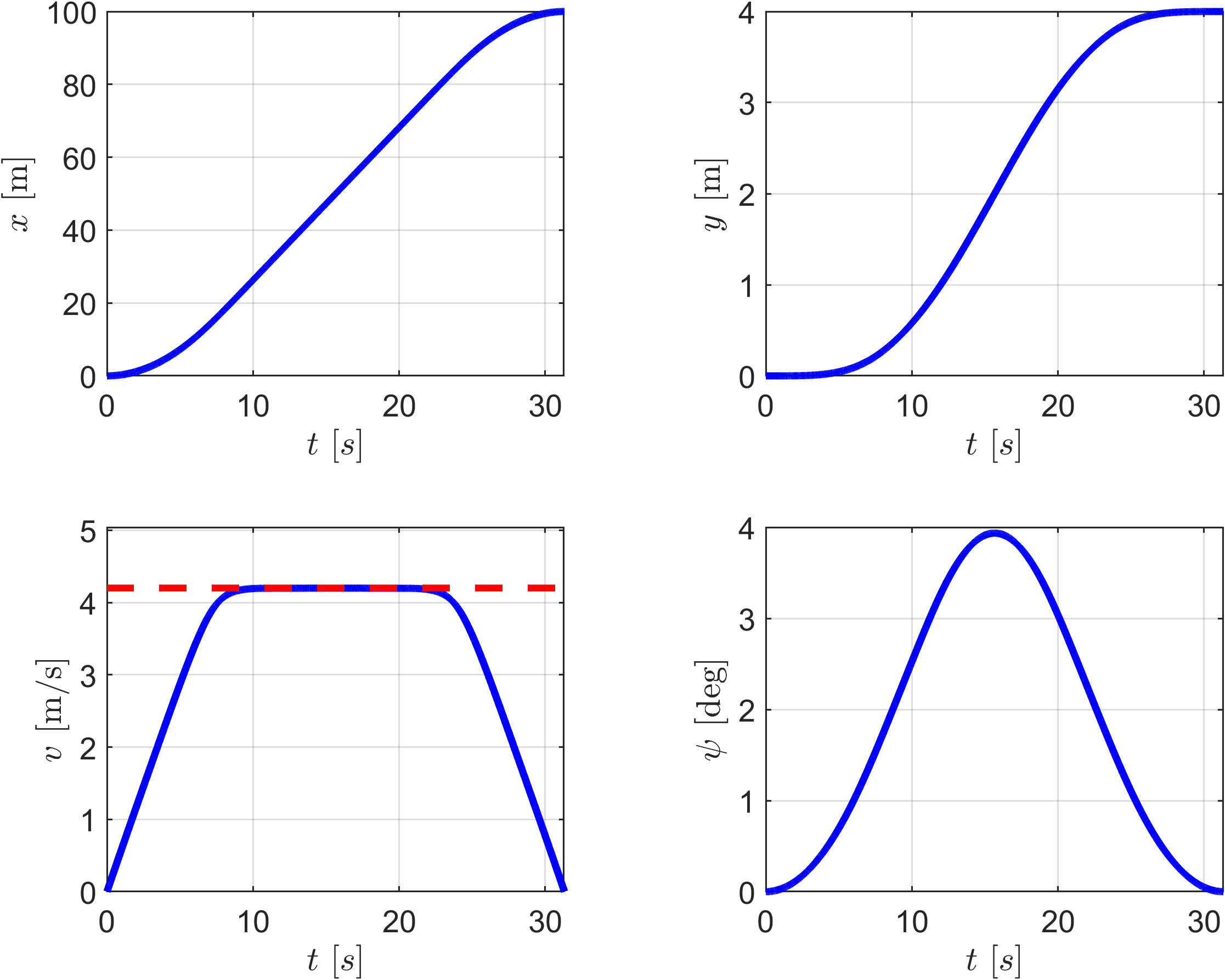}
    \includegraphics[width=4.3cm]{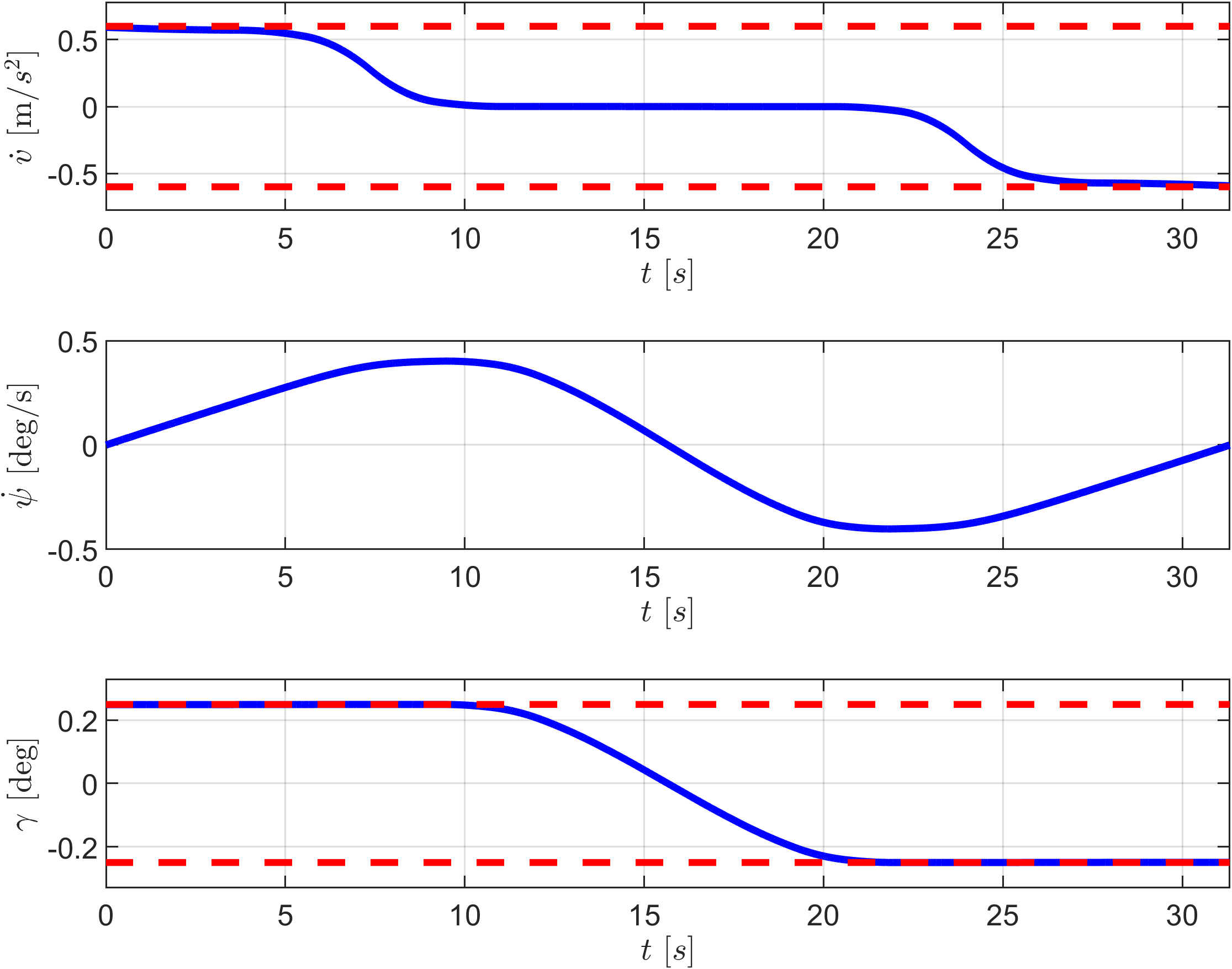}
    \caption{The state (left) and input (right) trajectories of Example \ref{example1} show that the required safety constraints are all satisfied for all $t \in [0,t_f]$: the speed respects $v(t) \leq 4.2$ m/s, the acceleration respects $|\dot{v}(t)| \leq 0.6 $ m/s$^2$ and the steering angle respects $|\gamma(t)|\leq 0.25$ deg}
    \label{fig:safe}
\end{figure}
\end{example}

\begin{example}\label{example2}
\begin{table}[!b]
    \centering
    \caption{Efficiency and optimality of optimal control solvers.}
    \begin{tabular}{|c|c|c|}
    \hline
    & \textbf{Avg solve time} [ms]  & \textbf{Objective value} [-]  \\
    \hline
    proposed method & 28.8 & 6.8495\\
    \hline
    ICLOCS2 & 257 & 6.8134\\
    \hline
    OpenOCL & 94.3 & 6.5534\\
    \hline
    \end{tabular}
    \label{tab:lane_change_perf}
\end{table}
We compare the performance and optimality of the proposed framework with two state-of-the-art optimal control solvers: ICLOCS2 \cite{nie2018iclocs2} and OpenOCL \cite{koenemann2017openocl}. For our framework, we solve the three SOCPs using YALMIP \cite{lofberg2004yalmip} with MOSEK \cite{mosek}. The vehicle considered is a 2021 Bolt EV by Chevrolet with a wheelbase length of $L=2.601$ meters. We let the duration penalty factor $\nu = 1$. We assume a typical 2-lane, straight, road with a posted speed limit of 40 miles per hour and consider a left lane change. The initial and final states are specified as $\vect{x}_0 = (0, 0, 16, 0)^T$ and $\vect{x}_f = (75, 3.7, 17.5, 0)^T$, respectively.
The other parameters are chosen as $\overline{\gamma} = 0.785$ ($45$ degrees), $\mathcal{D}=\mathbb{R}^2$, $\overline{v} = 19$ m/s, $\overline{a} = 2$ m/s$^2$ and $\nu = 1$. With these parameters, we solve the optimal control problem \eqref{OPT-CTRL} with the proposed framework, ICLOCS2 with analytical derivatives information provided and 40 discretization samples, and OpenOCL with the default configuration and 40 discretization samples. The resulting trajectories $\vect{x}(t)$ and $\vect{u}(t)$ are shown in Figure \ref{fig:lane_change}. 
To compare the computational burden of each algorithm, we collect an average solve time of 50 runs with each approach. The results are shown in Table \ref{tab:lane_change_perf}. 
In this example, the proposed approach achieved solve times nearly four times faster than the next leading method. In addition, the objective value of the proposed approach, while higher, was still comparable to that of the other two solvers. 
\end{example}


\begin{figure}[htp]
\centering
\includegraphics[width=8.5cm]{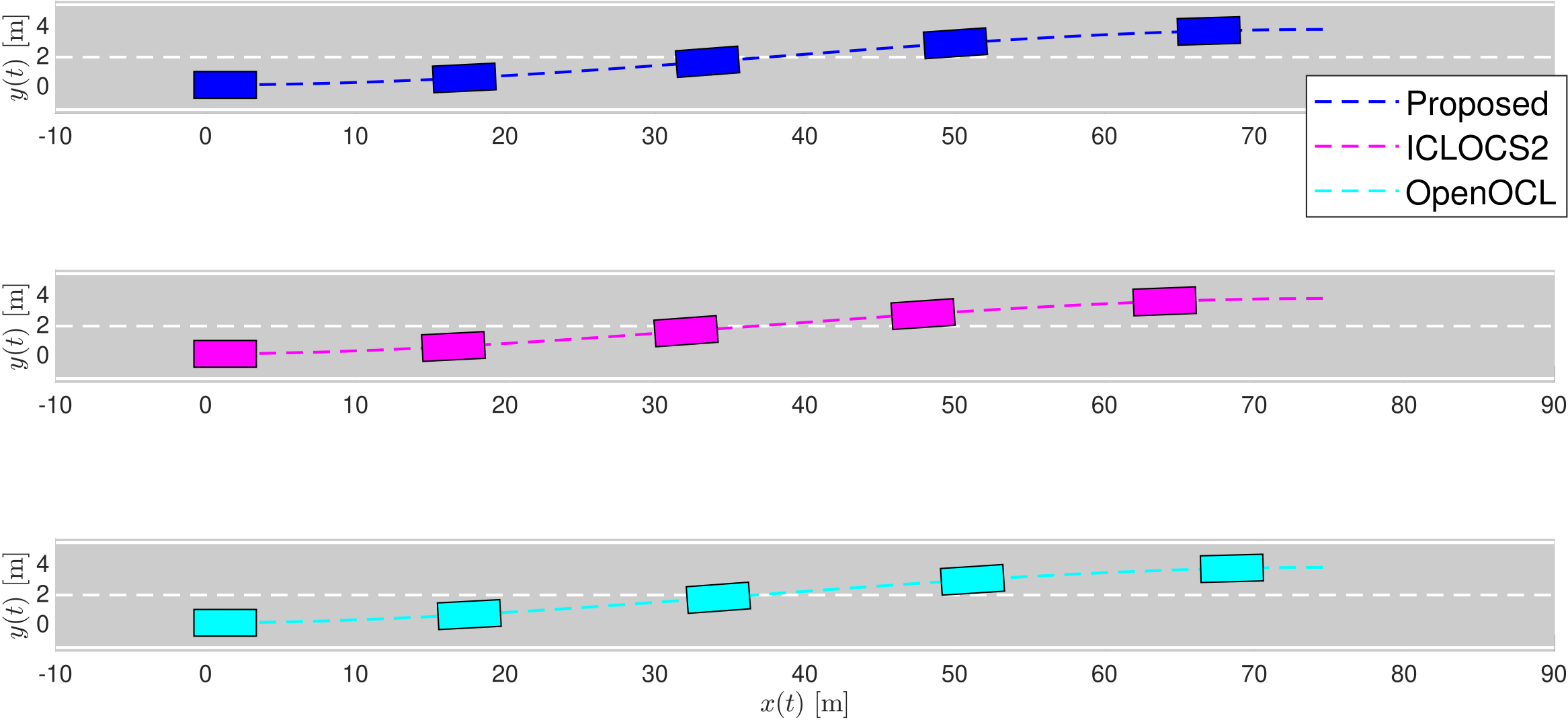}
\includegraphics[width=4.25cm]{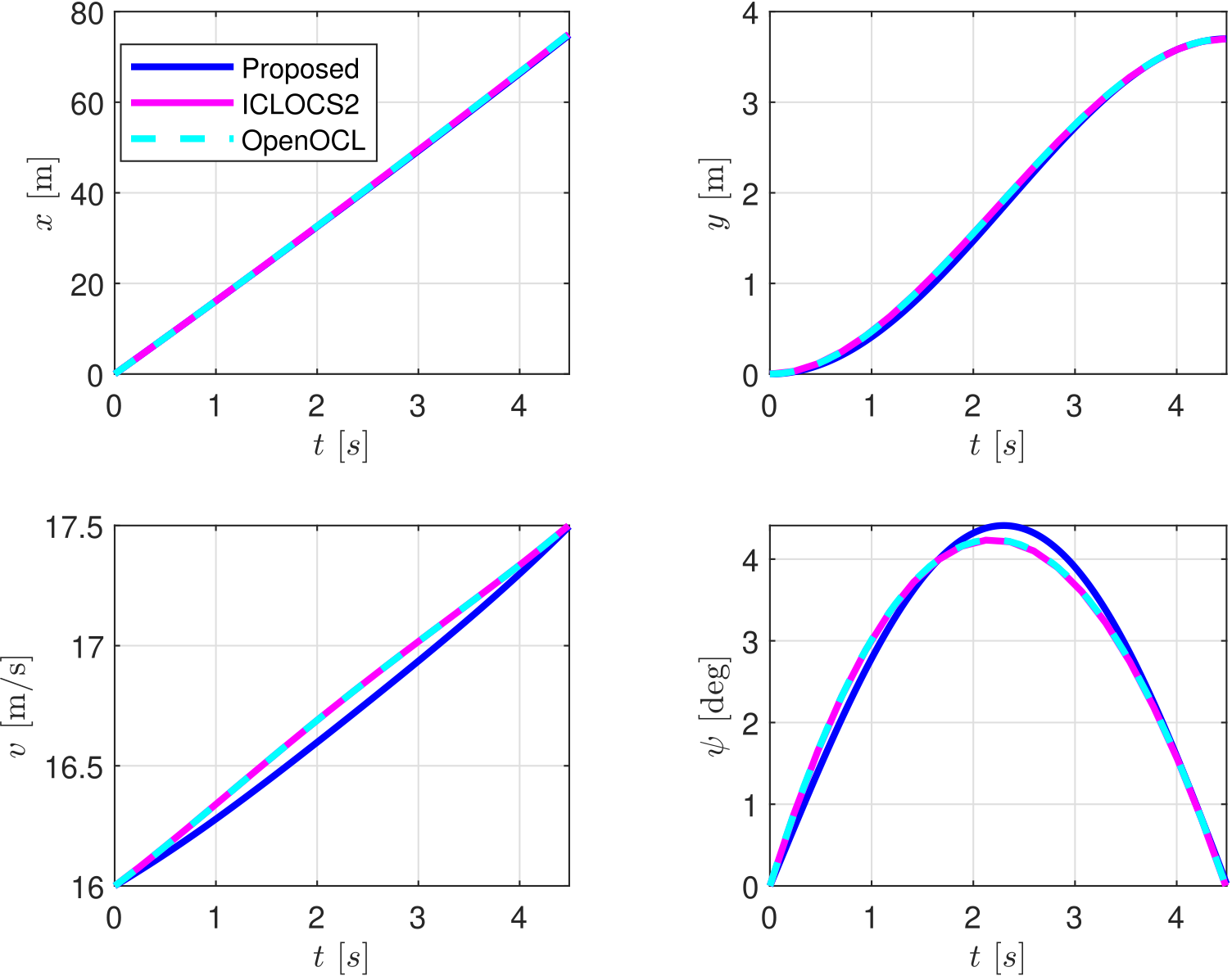}
\includegraphics[width=4.25cm]{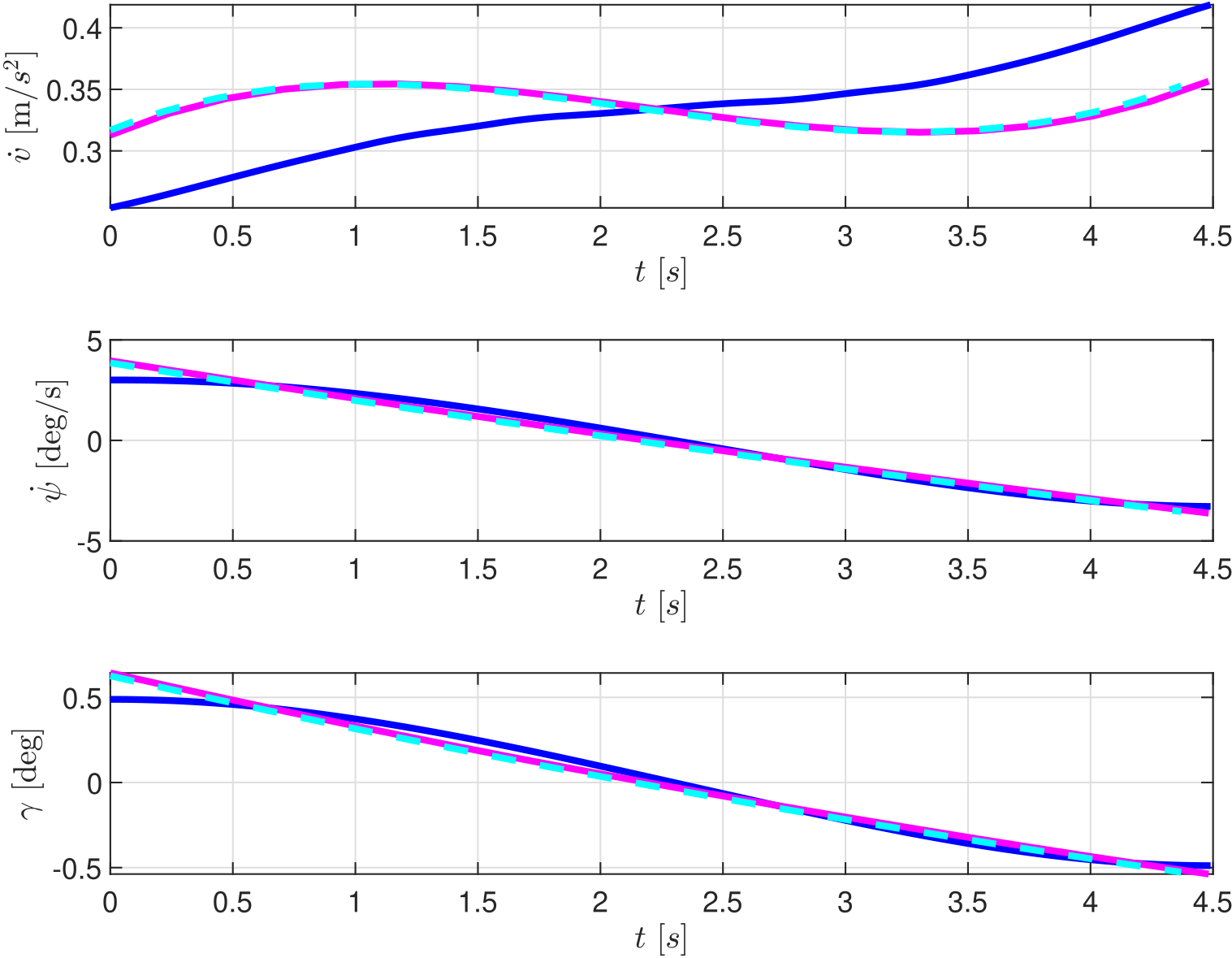}
\caption{Simulation results of Example \ref{example2}. The top figure shows the lane changing maneuver and the bottom plots show the state and input trajectories obtained with each solver.}
\label{fig:lane_change}
\end{figure}


\begin{figure*}[!ht]
\centering
\includegraphics[width=3.5cm]{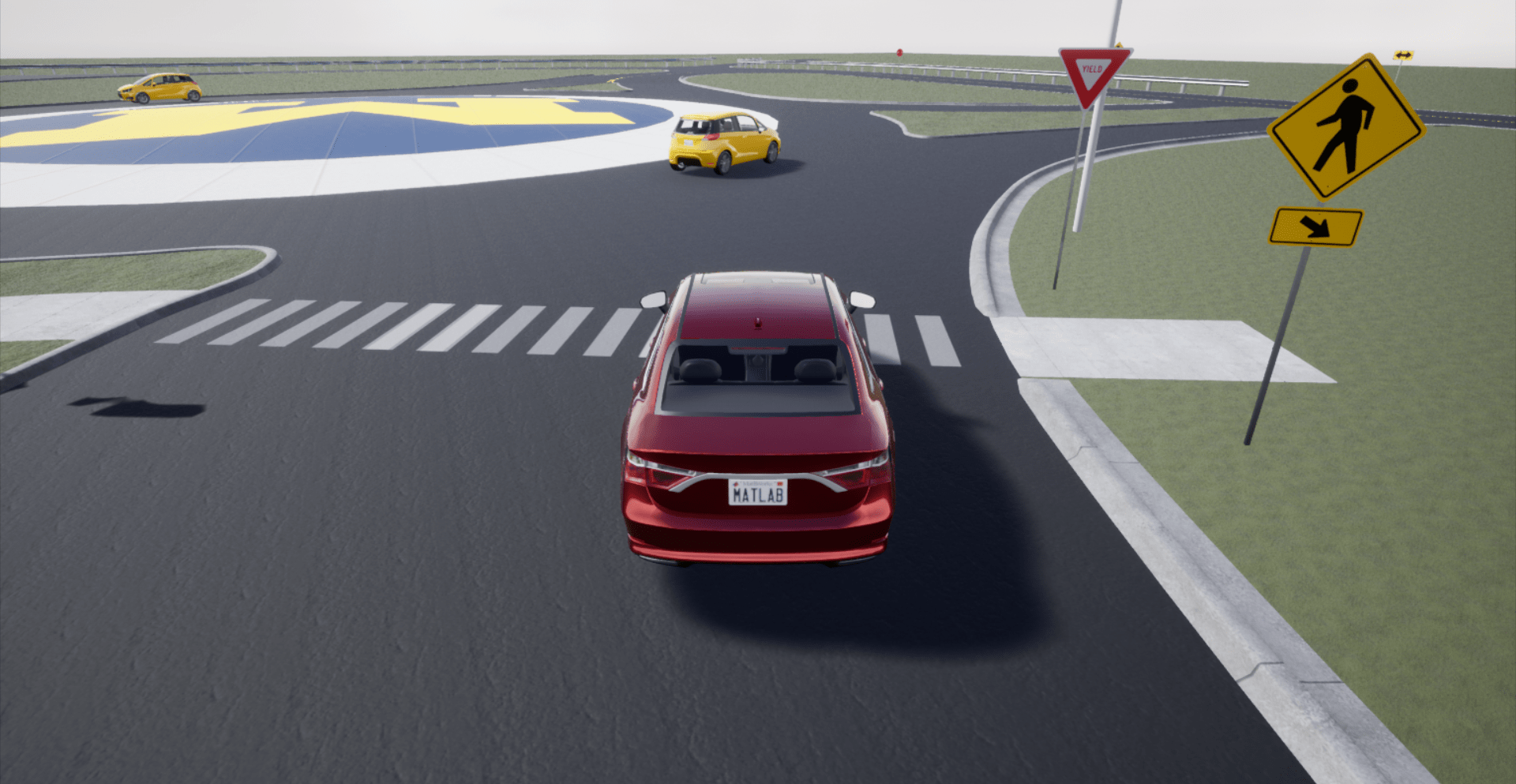}
\includegraphics[width=3.5cm]{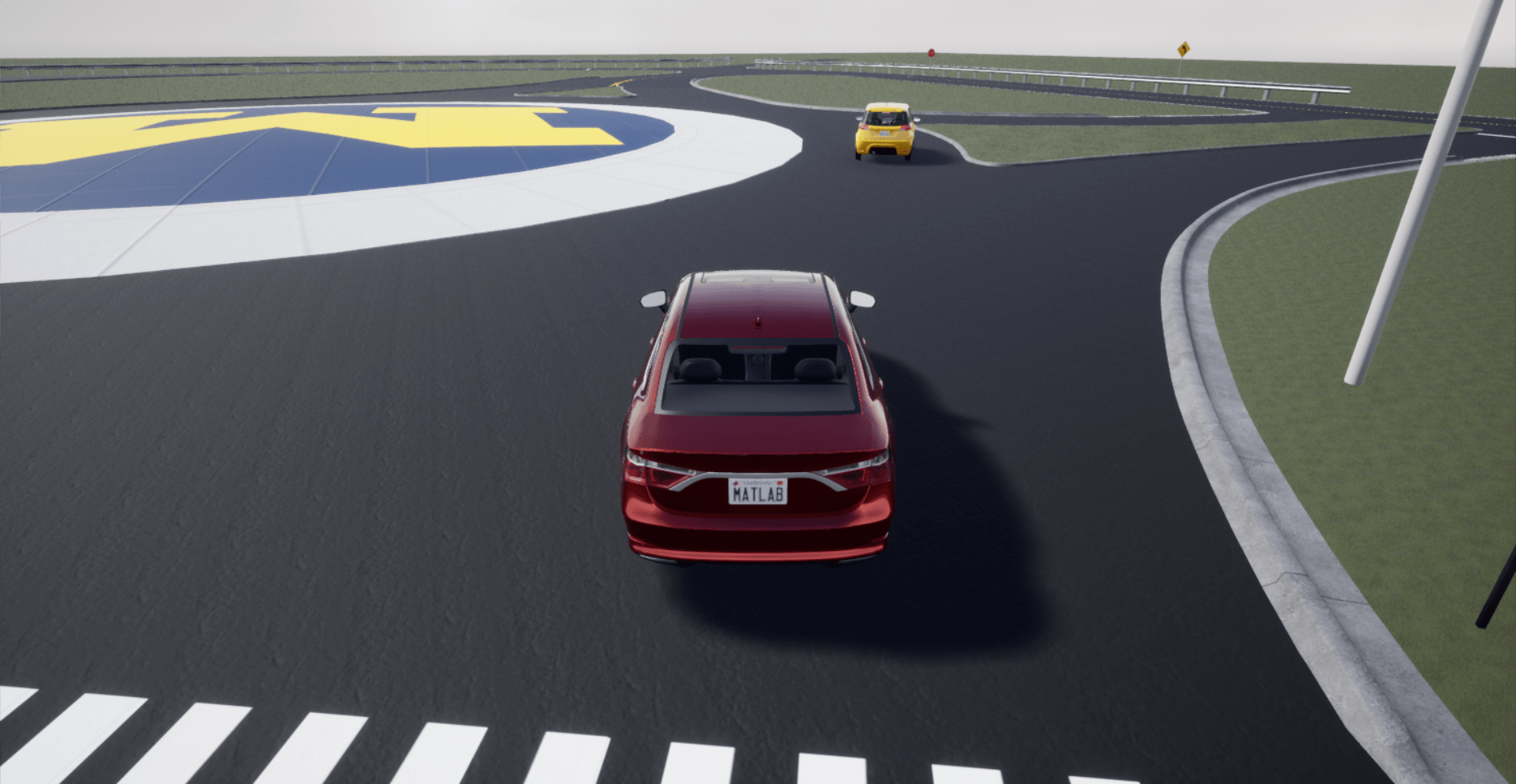}
\includegraphics[width=3.5cm]{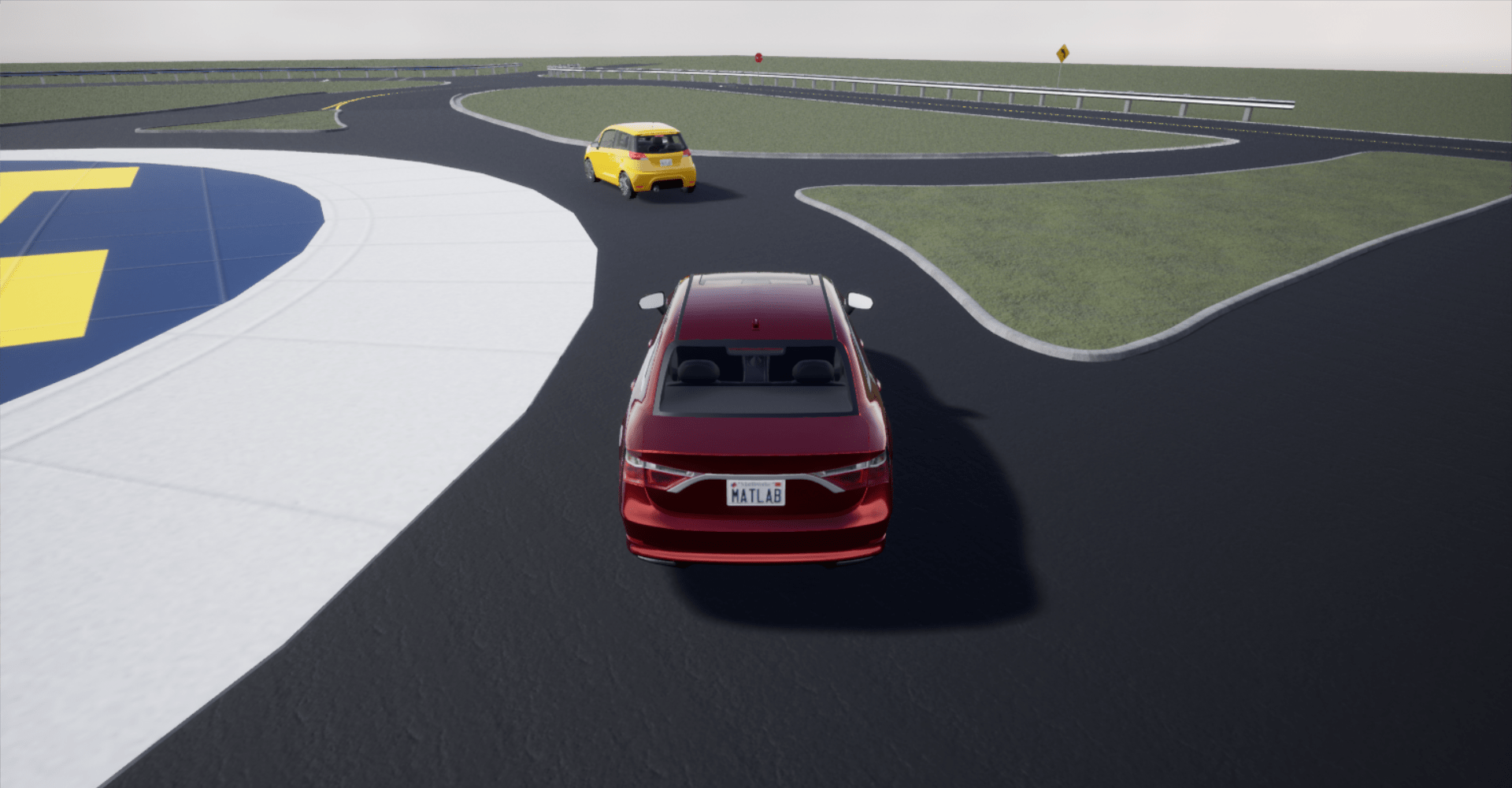}
\includegraphics[width=3.5cm]{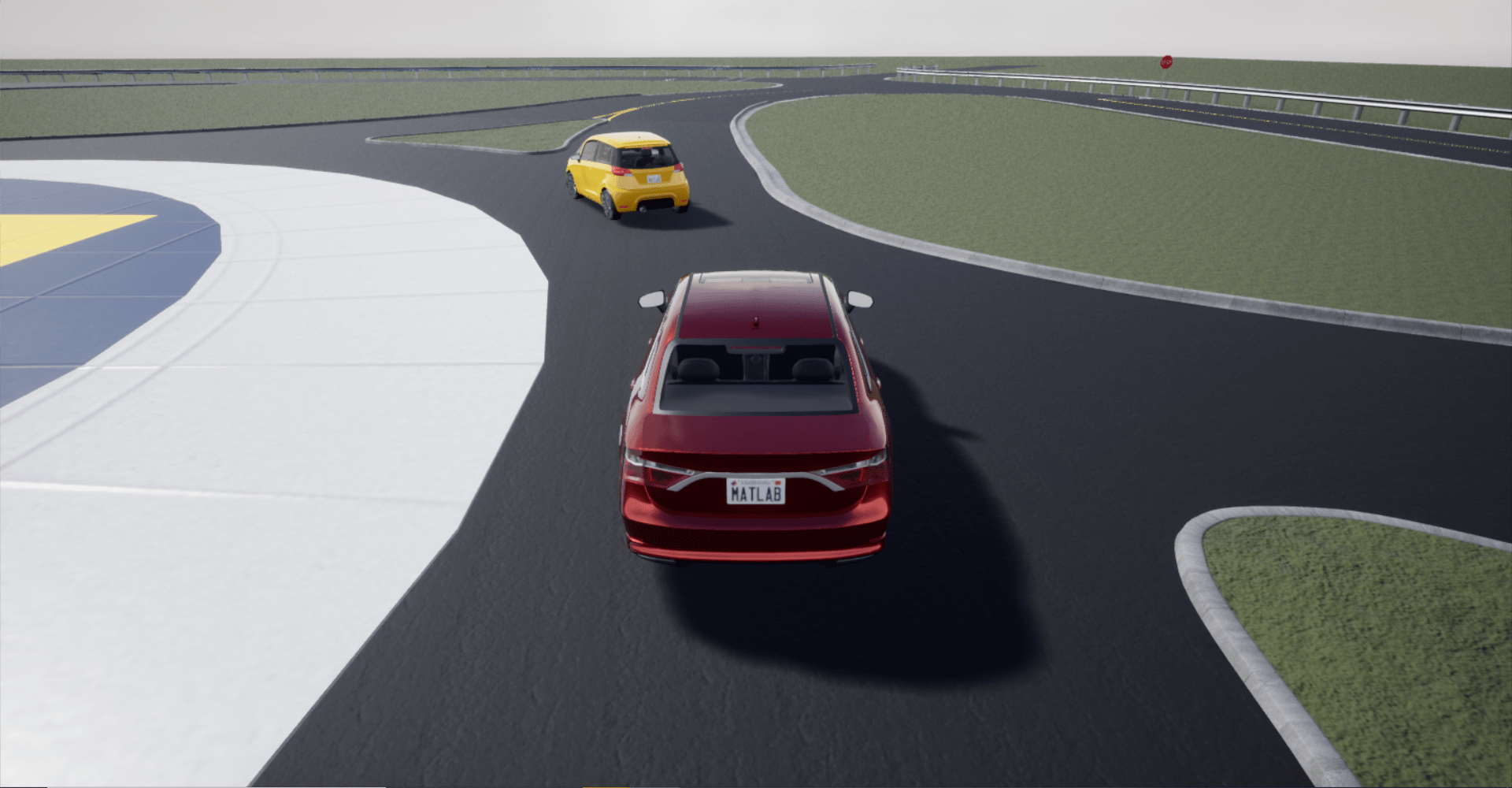}
\includegraphics[width=3.5cm]{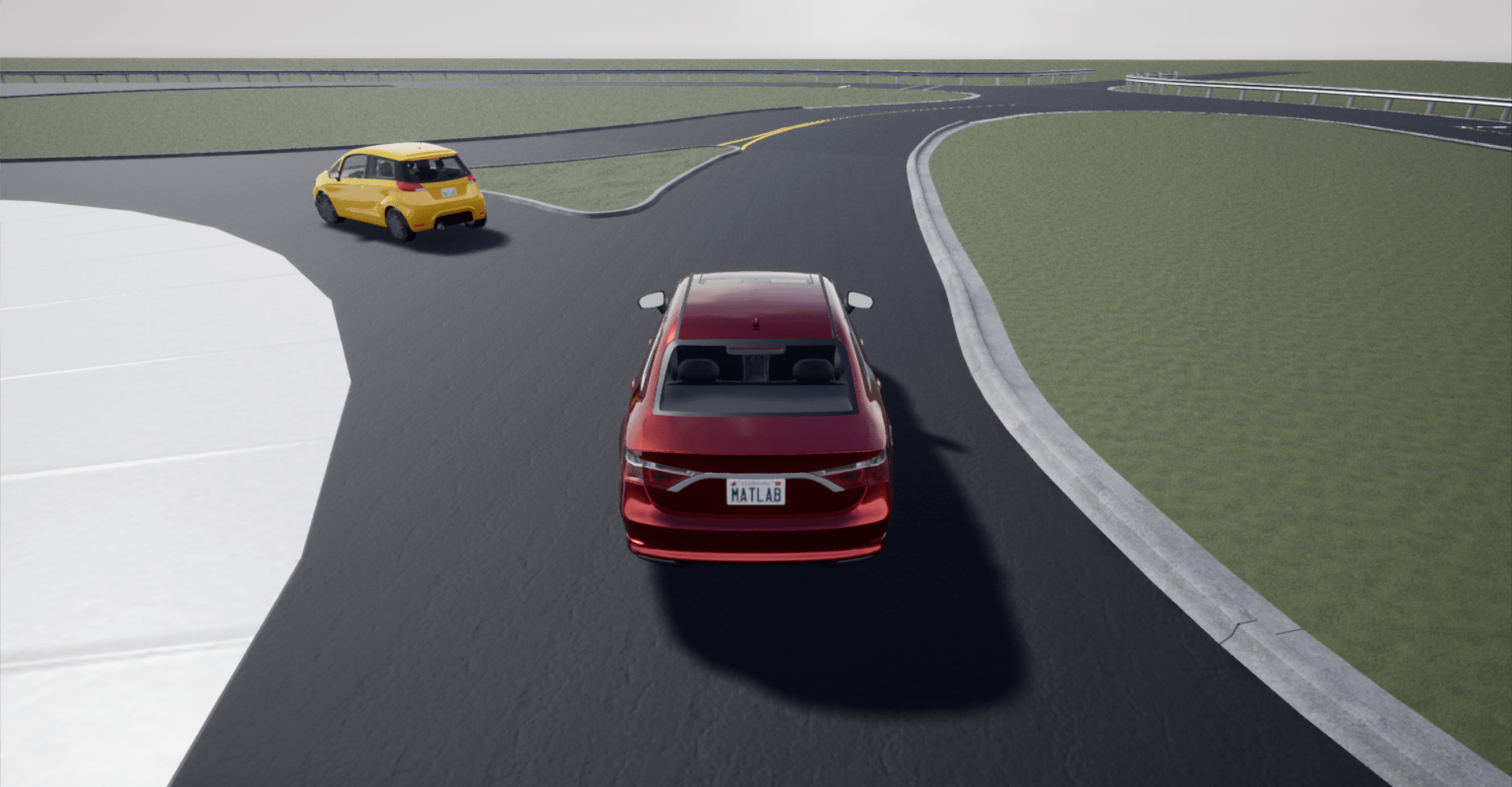}
\includegraphics[width=3.5cm]{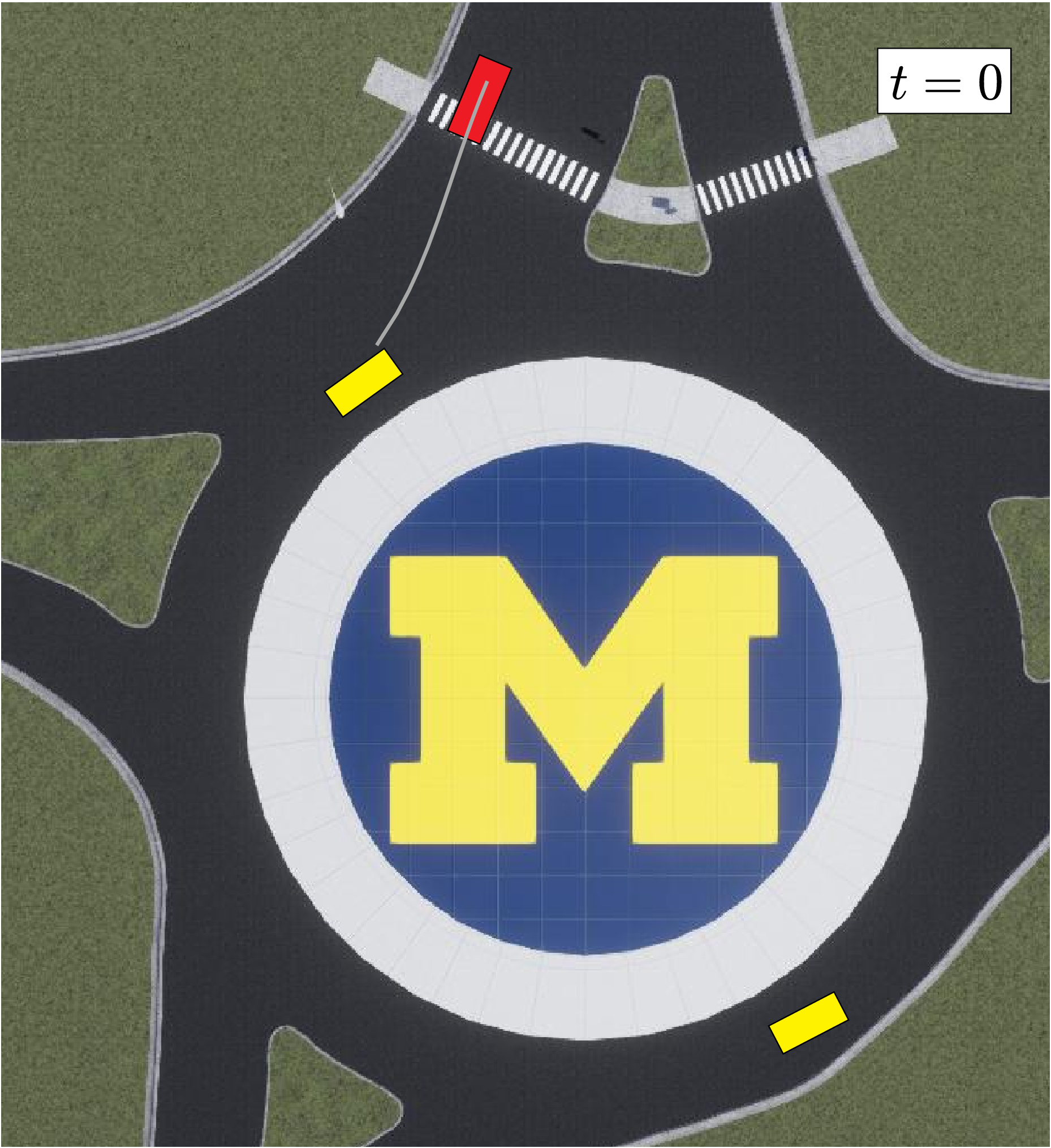}
\includegraphics[width=3.5cm]{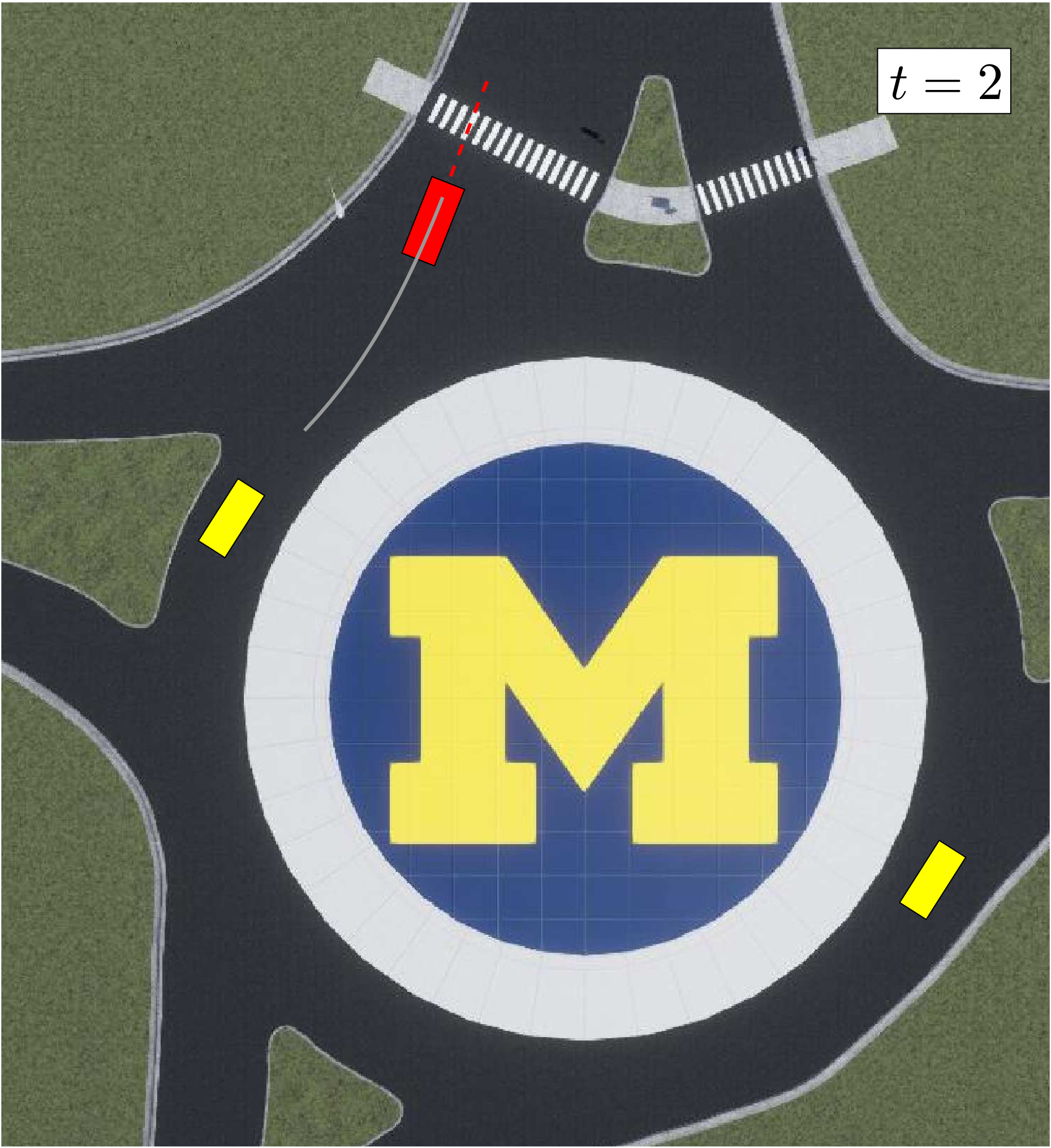}
\includegraphics[width=3.5cm]{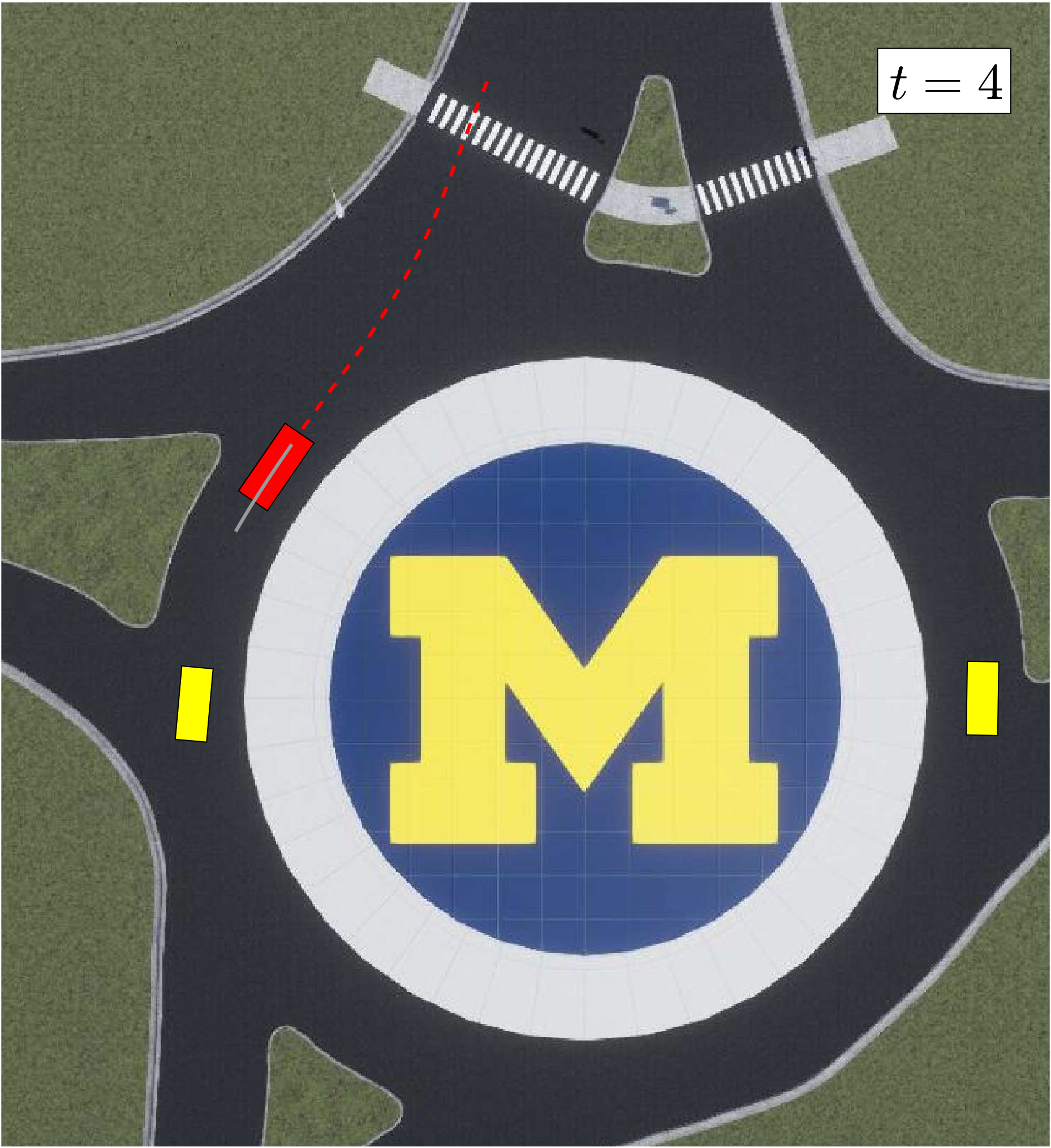}
\includegraphics[width=3.5cm]{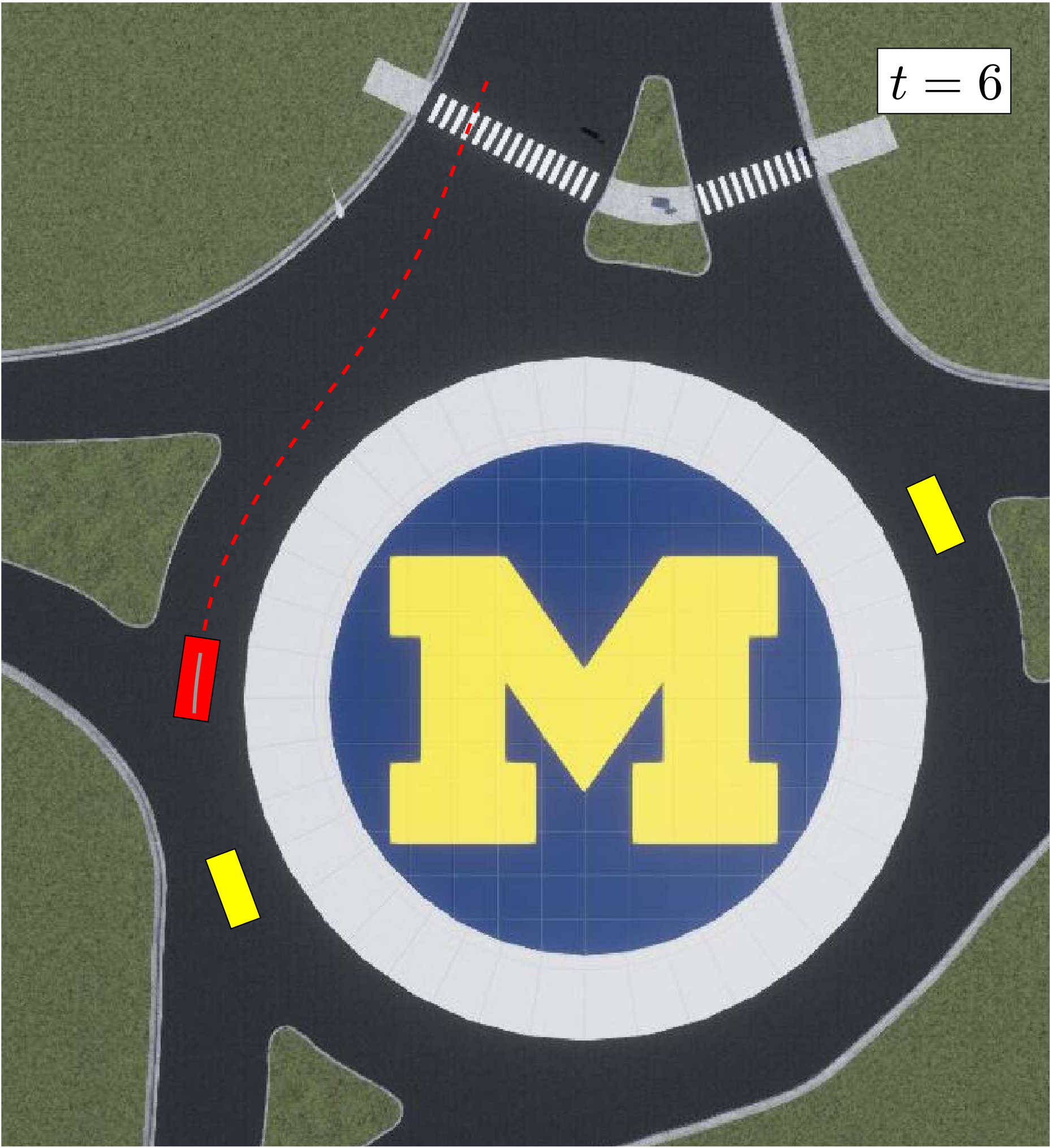}
\includegraphics[width=3.5cm]{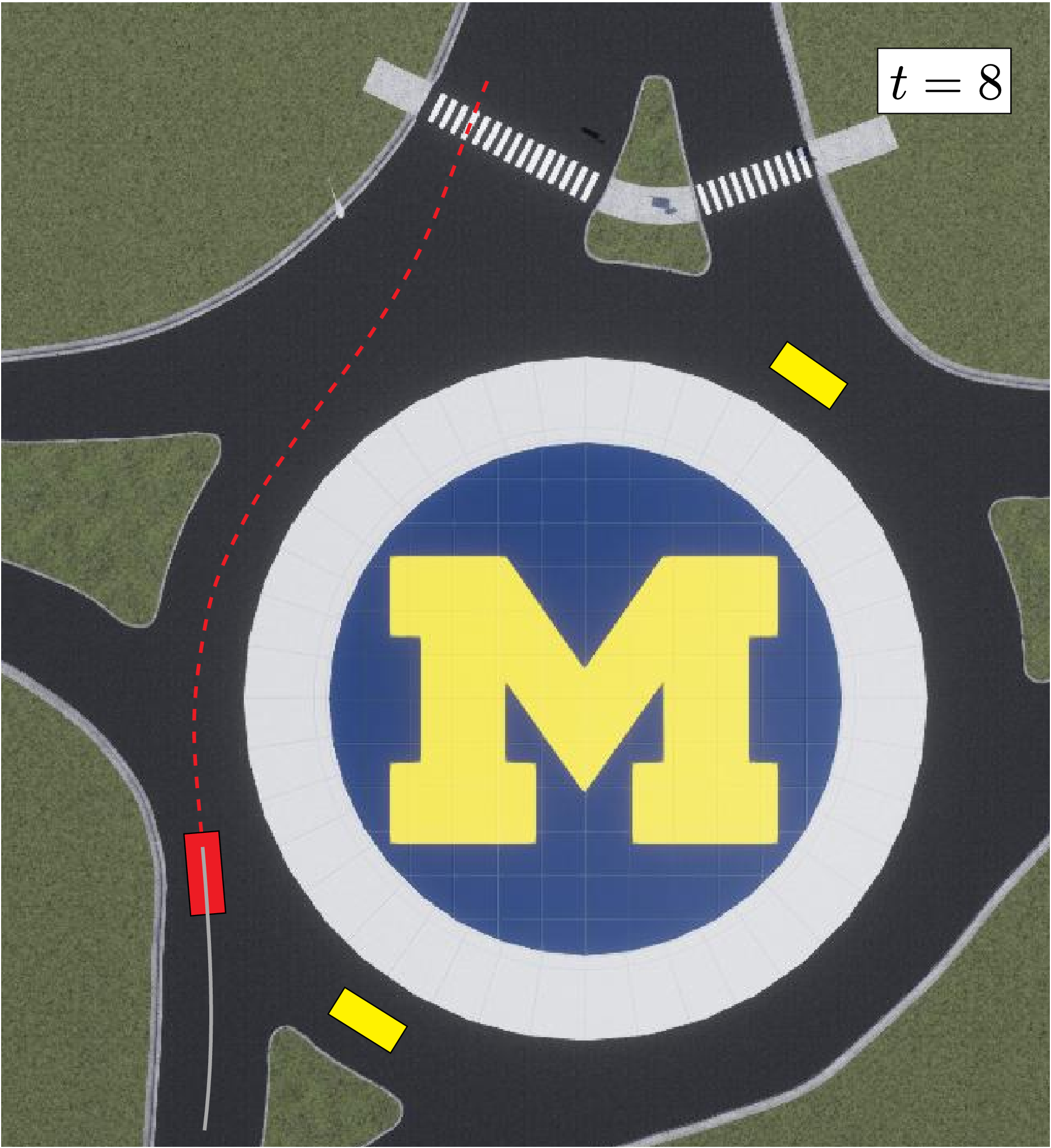}
\caption{Navigation using the proposed approach to solve \eqref{OPT-CTRL} at 25 Hz in Example \ref{example3}. The figure shows the scenario at times $t = 0, 2, 4, 6 ,8$ seconds.}
\label{fig:mcity}
\end{figure*}
\begin{example}\label{example3}
We demonstrate the real-time capabilities of the proposed approach in a simulated scenario located at Mcity's main roundabout.
The vehicle begins from rest at the roundabout entrance and must adjust to existing traffic and take the roundabout's second exit. We enforce the constraint $v(t) \leq \overline{v} = 11.176$ m/s to conform to typical 25 mph speed limits in residential areas. We also consider two actors driving around the roundabout with constant speed of $5$ m/s.
The adjustment to traffic is done by simple behavioral logic as follows: If the planned trajectory, with desired final position $\vect{r}_f$ located 15 meters ahead on the road, is obstacle-free, it is used; otherwise, if it collides with the vehicle in front, we adjust the endpoint of the trajectory to be $8$ meters behind the leading vehicle (but still on the road's center line) and enforce a final speed $v_f = 2.5 $ m/s.
The proposed framework is solved at 25 Hz and, for each solution, the corresponding control is applied. Thus, we achieve similar behavior as that of MPC. 
The scenario is rendered in MATLAB's 3D simulation environment powered by Unreal Engine. We show snapshots of the trajectory at five different time steps in Figure \ref{fig:mcity}. The figure also shows a bird's eye view of the scenario and the planned trajectory at the current time step (gray line).
\end{example}

\section{Conclusions} \label{sec:Conclusions}
In this work we presented a sequential SOCP approach to solve the optimal control problem with a kinematic bicycle model where the state and input trajectories generated are guaranteed to satisfy the constraints in the continuous-time sense. We also compared the performance of the proposed framework to that of state-of-the-art solvers and demonstrated its efficiency in a simulated scenario. In future work, we will leverage the safety guarantees of the proposed framework to achieve safety both at a higher level in the form of objective success and at a lower level in the form of tracking safety guarantees.

\bibliographystyle{IEEEtran}
\bibliography{IEEEabrv,mybib}

\end{document}